\documentclass{article}
\usepackage[pagewise]{lineno}

\usepackage{amsmath,amssymb,amsthm,mathrsfs}
\usepackage{hyperref,xcolor}
\usepackage{changepage}
\usepackage{enumitem} 
\usepackage{graphicx}
\usepackage{tikz}
\usepackage[utf8]{inputenc}
\usepackage[T1]{fontenc}
\usepackage{hyperref}
\usepackage{comment}
\allowdisplaybreaks

\numberwithin{equation}{section}

\newtheorem{myDefn}{Definition}[section]

\newtheorem{myProp}[myDefn]{Proposition}

\newtheorem{myRem}[myDefn]{Remark}

\newtheorem{myExa}[myDefn]{Example}

\newtheorem{myLem}[myDefn]{Lemma}

\newtheorem{myCor}[myDefn]{Corollary}

\newtheorem{myTheorem}[myDefn]{Theorem}

\DeclareMathOperator*{\argmin}{argmin}
\DeclareMathOperator*{\minn}{minimize}

\def\nn{\mathrm{n}}

\usepackage{indentfirst}

\def\R{\mathbb{R}}

\def\HH{\mathrm{H}}
\def\LL{\mathrm{L}}

\def\MD{\mathcal{D}}

\newcommand{\fonction}[5]{\begin{array}[t]{lrcl}#1 :&#2 &\longrightarrow &#3\\&#4& \longmapsto &#5 \end{array}}

\newcommand{\dual}[2]{\left\langle #1 , #2 \right\rangle}

\newlist{primenumerate}{enumerate}{1}
\setlist[primenumerate,1]{label={\roman*$'$}}

\title{Shape optimization for contact problem involving Signorini unilateral conditions}
\author{Aymeric Jacob de Cordemoy\footnote{Sorbonne Université, Université Paris Cité, CNRS, Laboratoire Jacques-Louis Lions, LJLL, F-75005 Paris, France.~\texttt{aymeric.jacob\_de\_cordemoy@sorbonne-universite.fr}}
}
\begin{document}

\maketitle

\begin{abstract}
    This paper investigates a shape optimization problem involving the Signorini unilateral conditions in a linear elastic model, without any penalization procedure. The shape sensitivity analysis is performed using tools from convex and variational analysis such as proximal operators and the notion of twice epi-differentiability. We prove that the solution to the Signorini problem admits a directional derivative with respect to the shape which moreover coincides with the solution to another Signorini problem. Then, the shape gradient of the corresponding energy functional is explicitly characterized which allows us to perform numerical simulations to illustrate this methodology.
\end{abstract}

\textbf{Keywords:} Shape optimization, shape sensitivity analysis, variational inequalities, contact mechanics, Signorini unilateral conditions, proximal operator, twice epi-differentiability.  

\medskip

\noindent \textbf{AMS Classification:} 49Q10, 49J40, 35J86, 74M15, 74P10.

%\tableofcontents

\section{Introduction}

\paragraph{Motivation.} On the one hand,  mechanical contact models are used to study
the deformation of solids that touch each other on parts of their boundaries. One of the mechanical setting consists in a deformable body which is in contact with a rigid
foundation without penetrating it and frictionless. From the mathematical point of view, the non-permeability conditions take the form of inequalities on the contact
surface called {\it Signorini unilateral conditions} (see, e.g.,~\cite{15SIG,16SIG}). Thus, those mechanical contact problems are usually investigated through the theory of variational inequalities, and the Signorini unilateral conditions cause nonlinearities in the corresponding variational formulations.
On the other hand, shape optimization is the mathematical field aimed at finding the optimal
shape of a given object for a given criterion, that is the shape which minimizes a certain cost
functional while satisfying given constraints. In order to numerically solve a shape optimization problem, the standard gradient descent method requires to compute the {\it shape gradient} of the cost functional.

Shape optimization problems with mechanical contact models involving Signorini unilateral conditions have been studied in the literature, and classical techniques to compute \textit{material} and \textit{shape derivatives} are based on Mignot’s theorem (see~\cite{MIGNOT}) about the \textit{conical differentiability} of projection operators on nonempty \textit{polyhedric} closed convex sets (see, e.g,~\cite{HARAUX,MAURY,SOKOZOL}). The material and shape derivatives are usually characterized with abstract variational inequalities, thus cause difficulties to compute a suitable shape gradient of the cost functional. These difficulties are usually solved in the literature using a penalization procedure (see, e.g.,~\cite{KIKUCHI}), which consists in adding a penalty functional in the optimization problem associated with the model, in order to handle the constraints due to the Signorini unilateral conditions. Hence, the optimality condition is described by a variational equality (see, e.g.,~\cite{MAUALLJOU,CHAUDET2,HINTERMULLERLAURAIN,LUFT}). However this penalization method does not take into account the exact characterization of the solution and may perturb the original nature of the model.

\paragraph{Novelties compared with existing literature.}
In this paper we investigate a shape optimization problem involving the Signorini unilateral conditions, using a new methodology based on tools from convex and variational analysis such as the notion of {\it proximal operator} introduced by J.J.\ Moreau in~1965 (see~\cite{MOR}) and the notion of {\it twice epi-differentiability} introduced by R.T.\ Rockafellar in 1985 (see~\cite{Rockafellar}). Note that we have studied the feasibility of this methodology on a shape optimization problem involving the {\it scalar Tresca friction law} (see~\cite{ABCJ}). First this new methodology allows us to recover the results obtained in~\cite{CHAUDET2},~\cite[Chapter 5 Section 5.2 p.111]{MAURY} and~\cite[Chapter 4 Section 4.6 p.205]{SOKOZOL}. Indeed, if a nonempty closed convex set is polyhedric, then from Mignot's theorem the projection operator on this set is conically differentiable, and its conical derivative coincides with the proximal operator associated with the \textit{second-order epi-derivative} of the appropriate indicator function, and thus our approach coincides with that used in the literature. Second the main novelty of the present work is that, under appropriate assumptions, our method permits to characterize the material and shape derivatives of the solution to the Signorini problem as the solutions to other Signorini problems. This point, to the best of our knowledge, has not been noticed in the literature. Furthermore, by using this new characterization, we obtain an explicit expression of the shape gradient of the corresponding energy functional. This shape gradient generalizes the one obtained in~\cite[Section 5.5]{Gilles} where the Signorini unilateral conditions are on a rectilinear boundary part. Therefore, without using any penalization procedure, the present work can be seen as a complement and an extension of the previous articles on this subject. 
\medskip

\paragraph{Description of the shape optimization problem and methodology.}

In this paragraph, we use standard notations which are recalled in Section~\ref{BVP}. Let~$d\in\{2,3\}$ which represents the dimension, $f$ be a function in $\HH^{1}(\R^{d},\R^d)$, $\Omega_{\mathrm{ref}}$ be a nonempty connected bounded open subset of~$\R^d$ with Lipschitz boundary $\Gamma_{\mathrm{ref}}:=\partial{\Omega_{\mathrm{ref}}}$, such that $\Gamma_{\mathrm{ref}}=\Gamma_{\mathrm{D}}\cup{\Gamma_{\mathrm{S}_{\mathrm{ref}}}}$, where $\Gamma_{\mathrm{D}}$ and ${\Gamma_{\mathrm{S}_{\mathrm{ref}}}}$ are two measurable pairwise disjoint subsets of $\Gamma_{\mathrm{ref}}$, and $\Gamma_{\mathrm{D}}$ has a positive measure.

In this paper we consider the shape optimization problem given by
\begin{equation}\label{shapeOptim}
    \minn\limits_{ \substack{ \Omega\in \mathcal{U}_{\mathrm{ref}} \\ \vert \Omega \vert = \vert \Omega_{\mathrm{ref}} \vert } } \; \mathcal{J}(\Omega),
\end{equation}
where
\begin{multline}\label{setadmiss}
     \mathcal{U}_{\mathrm{ref}} :=\biggl\{ \Omega\subset\R^{d} \mid \Omega \text{ nonempty connected bounded open subset of } \R^{d} \\  \text{ with Lipschitz boundary } \Gamma:=\partial{\Omega} \text{ such that } \Gamma_{\mathrm{D}}\subset\Gamma  \biggl\},
\end{multline}
 with the volume constraint $\vert \Omega \vert = \vert \Omega_{\mathrm{ref}} \vert > 0$, for all $\Omega\in\mathcal{U}_{\mathrm{ref}}$, and where $\mathcal{J} : \mathcal{U}_{\mathrm{ref}} \to \R$ is the \textit{Signorini energy functional} defined by
\begin{equation}\label{energysigno}
    \mathcal{J}(\Omega) := \frac{1}{2}\int_{\Omega} \mathrm{A}\mathrm{e}\left(u_\Omega\right):\mathrm{e}\left(u_\Omega\right)-\int_{\Omega}f\cdot u_{\Omega},
\end{equation}
where  $u_\Omega \in\HH^{1}_{\mathrm{D}}(\Omega,\R^{d})$ stands for the unique solution to the Signorini problem given by
\begin{equation}\label{Signoriniproblem2221}\tag{SP$_{\Omega}$}
\arraycolsep=2pt
\left\{
\begin{array}{rcll}
-\mathrm{div}(\mathrm{A}\mathrm{e}(u)) & = & f   & \text{ in } \Omega , \\
u & = & 0  & \text{ on } \Gamma_{\mathrm{D}} ,\\
\sigma_{\tau}(u) & = & 0  & \text{ on } \Gamma_{\mathrm{S}} ,\\
u_{\nn}\leq0, \sigma_{\nn}(u)\leq0 \text{ and } u_{\nn}\sigma_{\nn}(u) & = & 0  & \text{ on } \Gamma_{\mathrm{S}},
\end{array}
\right.
\end{equation}
where, for all~$\Omega \in \mathcal{U}_{\mathrm{ref}}$, $\Gamma:=\partial{\Omega}$,  $\Gamma_{\mathrm{S}}:=\Gamma\backslash\Gamma_{\mathrm{D}}$, and $\nn$ is the outward-pointing unit normal vector to~$\Gamma$.
In the linear elasticity model,~$\mathrm{A}$ is the stiffness tensor,~$\mathrm{e}$ is the infinitesimal strain tensor,~$\sigma_{\nn}$ is the normal stress,~$\sigma_{\tau}$ is the shear stress, and~$f$ models volume forces (see Section~\ref{BVP} for details). The normal boundary condition on~$\Gamma_{\mathrm{S}}$ is known as the Signorini
unilateral conditions which described the non-permeability of~$\Gamma_{\mathrm{S}}$ (that is~$u_\nn\leq0$), and that there are only compressive stresses exerted on~$\Gamma_{\mathrm{S}}$ (that is~$\sigma_{\nn}(u)\leq0$). Note that we focus here on minimizing the energy functional (as in~\cite{FULM,HPM, VELI}) which corresponds to maximize the compliance (see~\cite{ALL}).

For any~$\Omega \in \mathcal{U}_{\mathrm{ref}}$, the unique solution $u_{\Omega}$ {to}~\eqref{Signoriniproblem2221} satisfies
\begin{equation*}
\displaystyle\int_{\Omega}\mathrm{A}\mathrm{e}(u_\Omega):\mathrm{e}(v-u_\Omega)\geq\int_{\Omega}f\cdot (v-u_\Omega), \qquad \forall v\in\mathcal{K}^{1}(\Omega),
\end{equation*}
where $\mathcal{K}^{1}(\Omega)$ is the nonempty closed convex subset of $\HH^{1}_{\mathrm{D}}(\Omega,\R^{d})$ given by
$$
\mathcal{K}^{1}(\Omega) := \left\{v\in\HH^{1}_{\mathrm{D}}(\Omega,\R^{d}) \mid v_{\nn}\leq0 \text{ \textit{a.e.} on }\Gamma_{\mathrm{S}} \right \},
$$
and is characterized by~$u_\Omega=\mathrm{proj}_{\mathcal{K}^1(\Omega)}(F_\Omega)$, where~$F_\Omega\in\HH^{1}_{\mathrm{D}}(\Omega,\R^{d})$
is the unique solution to the Dirichlet-Neumann problem
\begin{equation*}
\arraycolsep=2pt
\left\{
\begin{array}{rcll}
-\mathrm{div}(\mathrm{A}\mathrm{e}(F)) & = & f   & \text{ in } \Omega , \\
F & = & 0  & \text{ on } \Gamma_{\mathrm{D}} ,\\
\mathrm{A}\mathrm{e}(F)\nn & = & 0  & \text{ on } \Gamma_{\mathrm{S}},
\end{array}
\right.
\end{equation*} 
and where $\mathrm{proj}_{\mathcal{K}^1(\Omega)}$ stands for the projection operator on $\mathcal{K}^1(\Omega)$. We refer for instance to~\cite{4ABC} for details on existence/uniqueness and characterization of the solution to Problem~\eqref{Signoriniproblem2221}.
In order to use our methodology, which is based, in particular, on the proximal operator, we characterize~$u_\Omega$ as (see Remark~\ref{projetprox})
$$
u_\Omega=\mathrm{prox}_{\iota_{\mathcal{K}^{1}(\Omega)}}(F_\Omega),
$$
where $\mathrm{prox}_{\iota_{\mathcal{K}^{1}(\Omega)}}$ is the proximal operator associated with the Signorini indicator function~$\iota_{\mathcal{K}^{1}(\Omega)}$, which is defined by $\iota_{\mathcal{K}^{1}(\Omega)}(v):=0$ if $v\in\mathcal{K}^{1}(\Omega)$, and~$\iota_{\mathcal{K}^{1}(\Omega)}(v):=+\infty$ otherwise. 
To deal with the numerical treatment of the above shape optimization problem, a suitable expression of the shape gradient of~$\mathcal{J}$ is required. For this purpose, we follow the classical strategy developed in the shape optimization literature (see, e.g.,~\cite{ALL,DAPO,HENROT,SOKOZOL}). Consider~$\Omega_{0} \in \mathcal{U}_{\mathrm{ref}}$ and a direction~$\theta\in \mathcal{C}_{\mathrm{D}}^{2,\infty}(\R^{d},\R^{d})$, where
\begin{equation}\label{direc12}
\mathcal{C}_{\mathrm{D}}^{2,\infty}(\R^{d},\R^{d}):=\left\{\theta\in\mathcal{C}^{2}(\R^{d},\R^{d})\cap\mathrm{W}^{2,\infty}(\R^{d},\R^{d}) \mid \theta=0 \text{ on } \Gamma_{\mathrm{D}}\right\}.
\end{equation}
Then, for any $t\geq0$ sufficiently small such that~$\mathrm{id}+t\theta$ is a~$\mathcal{C}^{2}$-diffeomorphism of $\R^{d}$, we denote by~$\Omega_{t}:=(\mathrm{id}+t\theta)(\Omega_{0}) \in \mathcal{U}_{\mathrm{ref}}$ and by~$u_{t} := u_{\Omega_t} \in\HH^{1}_{\mathrm{D}}(\Omega_{t},\R^d)$, where $\mathrm{id} :  \R^{d}\rightarrow \R^{d}$ stands for the identity operator. To get an expression of the shape gradient of~$\mathcal{J}$ at~$\Omega_0$ in the direction~$\theta$, the first step naturally consists in obtaining an expression of the derivative of the map~$t \in \R_+ \mapsto u_{t} \in\HH^{1}_{\mathrm{D}}(\Omega_{t},\R^d)$ at~$t=0$. To overcome the issue that~$u_{t}$ is defined on the moving domain~$\Omega_t$, the classical change of variables~$\mathrm{id}+t\theta$ is considered, and we prove that~$\overline{u}_{t}:=u_{t}\circ(\mathrm{id}+t\theta)\in\HH^{1}_{\mathrm{D}}(\Omega_{0},\R^d)$ is the unique solution to the parameterized variational inequality
\begin{equation*}
\displaystyle\int_{\Omega_0}\mathrm{J}_{t}\mathrm{A}\left[\nabla{\overline{u}_{t}}\left(\mathrm{I}+t\nabla{\theta}\right)^{-1}\right]:\nabla{\left(v-\overline{u}_{t}\right)}\left(\mathrm{I}+t\nabla{\theta}\right)^{-1}\\\geq\int_{\Omega_0}f_{t}\mathrm{J}_{t}\cdot\left(v-\overline{u}_t\right),
\end{equation*}
for all $v\in\mathrm{K}^{1}_{t}(\Omega_0):=\left\{v\in\HH^{1}_{\mathrm{D}}(\Omega_0,\R^d) \mid v\cdot(\mathrm{I}+t\nabla{\theta}^{\top})^{-1}\nn\leq0 \text{ \textit{a.e.} on }\Gamma_{\mathrm{S}_0} \right \}$, where $\nn$ refers now to the outward-pointing unit normal vector to $\Gamma_0$, $f_{t}:=f\circ(\mathrm{id}+t\theta)\in\HH^{1}(\R^{d},\R^d)$, $\mathrm{J}_{t}:=\mathrm{det}(\mathrm{I}+t\nabla{\theta})\in\LL^{\infty}(\R^{d},\R)$ is the Jacobian, $\nabla{\theta}$ stands for the standard Jacobian matrix of~$\theta$ and~$\mathrm{I}$ is the identity matrix of~$\R^{d \times d}$. Thus $\overline{u}_t=\mathrm{prox}_{\iota_{\mathrm{K}^{1}_{t}(\Omega_0)}}(F_t)\in\mathrm{K}^{1}_{t}(\Omega_0)$,
where $F_{t}\in\HH^{1}(\Omega_{0},\R^d)$ is the unique solution to the parameterized variational equality
\begin{equation*}
\displaystyle\int_{\Omega_0}\mathrm{J}_{t}\mathrm{A}\left[\nabla{F_t}\left(\mathrm{I}+t\nabla{\theta}\right)^{-1}\right]:\nabla{v}\left(\mathrm{I}+t\nabla{\theta}\right)^{-1}=\int_{\Omega_0}f_{t}\mathrm{J}_{t}\cdot v, \qquad \forall v\in\HH^1_{\mathrm{D}}(\Omega_0,\R^d),
\end{equation*}
and $\mathrm{prox}_{\iota_{\mathrm{K}^{1}_{t}(\Omega_0)}}$ is the proximal operator associated with the indicator function $\iota_{\mathrm{K}^{1}_{t}(\Omega_0)}$ considered on the Hilbert space $\HH^1_{\mathrm{D}}(\Omega_0,\R^d)$ endowed with a perturbed scalar product (see details in Subsection~\ref{setting}). Now, the next step is to obtain an expression of the derivative of the map~$t \in \R_+ \mapsto \overline{u}_{t} \in \HH^{1}_{\mathrm{D}}(\Omega_0,\R^d)$ at~$t=~0$, which will be denoted by~$\overline{u}'_0 \in \HH^{1}_{\mathrm{D}}(\Omega_0,\R^d)$ and called material derivative. Since, for all $t\geq0$, $\overline{u}_t \in \mathrm{K}^{1}_{t}(\Omega_0)$, the main difficulty to compute $\overline{u}'_0$ is that~$\mathrm{K}^{1}_{t}(\Omega_0)$ depends on the parameter $t\geq0$. To overcome it, we  consider, for any $t\geq0$ sufficiently small, a second change of variables $(\mathrm{I}+t\nabla{\theta})^{-1}$ and we characterize $\overline{\overline{u}}_t:=(\mathrm{I}+t\nabla{\theta})^{-1}\overline{u}_{t}$ as $\overline{\overline{u}}_t=\mathrm{prox}_{\iota_{\mathcal{K}^{1}(\Omega_0)}}(E_{t})\in\mathcal{K}^{1}(\Omega_0)$ where $E_t\in \HH^{1}_{\mathrm{D}}(\Omega_0,\R^d)$ is the unique solution to a specific variational equality (see details in Subsection~\ref{setting}). Since~$\mathcal{K}^{1}(\Omega_0)$ is independant on $t\geq0$, we can now focus on the derivative of the map~$t \in \R_+ \mapsto \overline{\overline{u}}_{t} \in \HH^{1}_{\mathrm{D}}(\Omega_0,\R^d)$, denoted by $\overline{\overline{u}}'_0\in\HH^{1}_{\mathrm{D}}(\Omega_0,\R^d)$.

To deal with the differentiability (in a generalized sense) of the proximal operator~$\mathrm{prox}_{\iota_{\mathcal{K}^{1}(\Omega_0)}} : \HH^{1}_{\mathrm{D}}(\Omega_0,\R^d) \to \HH^{1}_{\mathrm{D}}(\Omega_0,\R^d)$, we invoke the notion of twice epi-differentiability for convex functions (see~\cite{Rockafellar} or Definition~\ref{epidiff}), which leads to the \textit{protodifferentiability} of the corresponding proximal operators. Precisely, using a result proved by C. N. Do about the twice epi-differentiability of a indicator function (see~\cite[Chapter 2, Example 2.10 p.287]{DO} or Lemma~\ref{epipoly}), we characterize~$\overline{\overline{u}}'_0\in\HH^{1}_{\mathrm{D}}(\Omega_0,\R^d)$ as
$
\overline{\overline{u}}'_0=\mathrm{prox}_{\mathrm{d}_{e}^{2}\iota_{\mathcal{K}^{1}(\Omega_0)}(u_0|E_0-u_0)}(E'_0),
$
where $\mathrm{d}_{e}^{2}\iota_{\mathcal{K}^{1}(\Omega_0)}(u_0|E_0-u_0)$ is the second-order
epi-derivative of $\iota_{\mathcal{K}^{1}(\Omega_0)}$ at $u_0$ for $E_0-u_0$, and~$E'_{0} \in \HH^{1}_{\mathrm{D}}(\Omega_0,\R^d)$ is the derivative of the map~$t \in \R_+ \mapsto E_{t} \in \HH^{1}_{\mathrm{D}}(\Omega_0,\R^d)$ at $t=0$. Then one deduces that
$$
\overline{u}'_0=\mathrm{prox}_{\mathrm{d}_{e}^{2}\iota_{\mathcal{K}^{1}(\Omega_0)}(u_0|E_0-u_0)}(E'_0)+\nabla{\theta}u_0,
$$ and we characterize it as the unique solution to a variational inequality, which is next used to obtain the shape gradient of $\mathcal{J}$.

Let us emphasize that, in this paper, we do not prove theoretically the existence of a solution to the shape optimization problem~\eqref{shapeOptim}. The interested reader can find some related existence results (for very specific geometries in the two dimensional case) in~\cite{HASKLAR}.

\paragraph{Main theoretical results.}
We summarize here our main theoretical results (given in Theorems~\ref{materialderiv1} and~\ref{shapederivofJsigno1}).
However, to make their expressions more explicit and elegant, we present them
under certain additional assumptions, within the framework of Corollaries~\ref{materialderiv2},~\ref{shapederiv1} and~\ref{shapederivofJ}, making them more suitable for this introduction.

\begin{enumerate}
  \item[(i)]  Under some appropriate assumptions described in Corollary~\ref{materialderiv2}, the map
$t\in\R_{+} \mapsto \overline{u}_{t} \in \HH^{1}_{\mathrm{D}}(\Omega_0,\R^d)$ is differentiable at $t=0$, and the material derivative $\overline{u}'_{0}\in\HH^{1}_{\mathrm{D}}(\Omega_0,\R^d)$ is the unique weak solution to the Signorini problem
\begin{equation*}
\footnotesize
{\arraycolsep=2pt
\left\{
\begin{array}{rcll}
-\mathrm{div}(\mathrm{A}\mathrm{e}(\overline{u}'_0)) & = & \ell(\theta)  & \text{ in } \Omega_0 , \\
\overline{u}'_0 & = & 0  & \text{ on } \Gamma_{\mathrm{D}} ,\\
\sigma_{\tau}(\overline{u}'_0) & = & h^m(\theta)_{\tau}  & \text{ on } \Gamma_{\mathrm{S}_0} ,\\
 \sigma_{\nn}(\overline{u}'_0) & = & h^m(\theta)_{\nn}  & \text{ on } \Gamma_{\mathrm{S}^{{u_0}_{\nn}}_{0,\mathrm{N}}} ,\\
{\overline{u}'_{0}}_{\nn}& = & \left(\nabla{\theta}u_0\right)_{\nn}  & \text{ on } \Gamma_{\mathrm{S}^{{u_0}_{\nn}}_{0,\mathrm{D}}} ,\\
{\overline{u}'_{0}}_{\nn}\leq\left(\nabla{\theta}u_0\right)_{\nn}, \sigma_{\nn}(\overline{u}'_0)\leq h^m(\theta)_{\nn} \text{ and } \left( {\overline{u}'_{0}}_{\nn}-\left(\nabla{\theta}u_0\right)_{\nn}\right)\left(\sigma_{\nn}(\overline{u}'_0)-h^m(\theta)_{\nn}\right) & = & 0  & \text{ on } \Gamma_{\mathrm{S}^{{u_0}_{\nn}}_{0,\mathrm{S}}},
\end{array}
\right.}
\end{equation*}
where:
$$
\bullet~\ell(\theta)\\=-\mathrm{div}(\mathrm{A}\mathrm{e}(\nabla{u_0}\theta))\in\LL^2(\Omega_0,\R^d);
$$ 
\bigskip
$$
\bullet~h^m(\theta):=((\mathrm{A}\mathrm{e}(u_0))\nabla{\theta}^{\top}+\mathrm{A}(\nabla{u_0}\nabla{\theta})-\sigma_{\nn}(u_0)(\mathrm{div}(\theta)\mathrm{I}+\nabla{\theta}^{\top}))\nn\in\LL^2(\Gamma_{\mathrm{S}_0},\R^d);
$$ 
and $\Gamma_{\mathrm{S}_0}$ is decomposed, up to a null set, as $\Gamma_{\mathrm{S}^{{u_0}_{\nn}}_{0,\mathrm{N}}}\cup\Gamma_{\mathrm{S}^{{u_0}_{\nn}}_{0,\mathrm{D}}}\cup\Gamma_{\mathrm{S}^{{u_0}_{\nn}}_{0,\mathrm{S}}}$ (see details in Corollary~\ref{materialderiv2}).

\item[(ii)]  We deduce in Corollary~\ref{shapederiv1} that, under appropriate assumptions, the shape derivative, defined by~$u'_{0}:=\overline{u}'_0-\nabla{u_0}\theta\in\HH^{1}_{\mathrm{D}}(\Omega_0,\R^d)$ (which corresponds, roughly speaking, to the derivative of the map $t\in\R_{+} \mapsto u_{t} \in \HH^{1}_{\mathrm{D}}(\Omega_t,\R^d)$ at~$t=0$), is the unique weak solution to the Signorini problem
\begin{equation*}
\small{\arraycolsep=2pt
\left\{
\begin{array}{rcll}
-\mathrm{div}(\mathrm{A}\mathrm{e}(u'_0)) & = & 0   & \text{ in } \Omega_0 , \\
u'_0 & = & 0  & \text{ on } \Gamma_{\mathrm{D}} ,\\
\sigma_{\tau}(u'_0) & = & h^s(\theta)_{\tau}  & \text{ on } \Gamma_{\mathrm{S}_0} ,\\
 \sigma_{\nn}(u'_0) & = & h^s(\theta)_{\nn}  & \text{ on } \Gamma_{\mathrm{S}^{{u_0}_{\nn}}_{0,\mathrm{N}}} ,\\
 {u'_{0}}_{\nn}& = & W(\theta)_{\nn} & \text{ on } \Gamma_{\mathrm{S}^{{u_0}_{\nn}}_{0,\mathrm{D}}} ,\\
{u'_{0}}_{\nn}\leq W(\theta)_{\nn}, \sigma_{\nn}(u'_0)\leq h^s(\theta)_{\nn} \text{ and } \left( {u'_{0}}_{\nn}-W(\theta)_{\nn} \right)\left(\sigma_{\nn}(u'_0)-h^s(\theta)_{\nn}\right) & = & 0  & \text{ on } \Gamma_{\mathrm{S}^{{u_0}_{\nn}}_{0,\mathrm{S}}},
\end{array}
\right.}
\end{equation*}
where:
$$\bullet~W(\theta):=\left(\nabla{\theta}u_0\right)-\left(\nabla{u_0}\theta\right)\in\HH^{1/2}(\Gamma_0,\R^d);
$$ 
\begin{multline*}
\bullet~h^s(\theta):=\theta \cdot \nn\left(\partial_{\nn}\left(\mathrm{A}\mathrm{e}(u_0)\nn\right)-\partial_{\nn}\left(\mathrm{A}\mathrm{e}(u_0)\right)\nn\right)+\mathrm{A}\mathrm{e}(u_0)\nabla_{\tau}\left(\theta\cdot\nn\right)-\nabla{\left(\mathrm{A}\mathrm{e}(u_0)\nn\right)}\theta\\-\sigma_{\nn}(u_0)\left(\mathrm{div}_{\tau}(\theta)\mathrm{I}+\nabla{\theta}^{\top}\right)\nn\in\LL^2(\Gamma_{\mathrm{S}_0},\R^d).
\end{multline*}
\item[(iii)] Finally the two previous items are used to obtain Corollary~\ref{shapederivofJ} asserting that, under appropriate assumptions, the shape gradient of $\mathcal{J}$ at~$\Omega_{0}$ in the direction~$\theta$ is given by
$$ 
\mathcal{J}'(\Omega_{0})(\theta)=\int_{\Gamma_{\mathrm{S}_{0}}} \left(\theta\cdot\nn\left(\frac{\mathrm{A}\mathrm{e}(u_0):\mathrm{e}(u_0)}{2}-f\cdot u_{0}\right)+\mathrm{A}\mathrm{e}(u_0)\nn\cdot\left(\nabla{\theta}u_0-\nabla{u_0}\theta\right)\right).
$$
One can notice that $\mathcal{J}'(\Omega_0)$ depends only on $u_0$ (and not on $u’_0$), thus its expression is explicit and linear with respect to the direction $\theta$ (see Remark~\ref{remarkadjoint} for details) and allows us to exhibit a descent direction of $\mathcal{J}$. Hence, using this descent direction together with a basic Uzawa algorithm to take into account the volume constraint, we perform in Section~\ref{numericalsim} numerical simulations to solve the shape optimization problem~\eqref{shapeOptim} on a two-dimensional example.

\end{enumerate}

\paragraph{Organization of the paper.}
The paper is organized as follows. Section~\ref{PPrelimi} is the preliminary section: in Subsection~\ref{rappelconvex} some reminders on proximal operator and twice epi-differentiability are introduced and, in Subsection~\ref{BVP} , we describe the functional framework and recall some classical boundary value problems involved all along the paper. Section~\ref{mainresultoff} is
the core of the present work where the main theoretical results are stated and proved.
In Section~\ref{numericalsim}, numerical simulations are performed to solve the shape optimization problem~\eqref{shapeOptim} on a two-dimensional example. Finally Appendix~\ref{rappelgeneral} is dedicated to some
basic recalls from capacity theory and differential geometry.

\section{Preliminaries}\label{PPrelimi}

In this section we start in Subsection~\ref{rappelconvex} with some notions from convex and variational analysis such as the proximal operator and the twice epi-differentiability which are essential in this paper. In Subsection~\ref{BVP} the functional framework is presented and we recall some classical boundary value problems used throughout the paper.
 
\subsection{Reminders on proximal operator and twice epi-differentiability}\label{rappelconvex}
For notions and results presented in this section, we refer to standard references such as~\cite{BREZ2,MINTY,ROCK2} and~\cite[Chapter~12]{ROCK}. In the sequel $(\mathcal{H}, \dual{\cdot}{\cdot}_{\mathcal{H}})$ stands for a general real Hilbert space.

\begin{myDefn}[Domain and epigraph]
Let  $\phi \,: \, \mathcal{H}\rightarrow \mathbb{R}\cup\left\{\pm \infty \right\}$.
The domain and the epigraph of~$\phi$ are respectively defined by
$$
\mathrm{dom}\left(\phi\right):=\left\{x\in \mathcal{H} \mid \phi(x)<+\infty \right\} \quad \text{and} \quad
\mathrm{epi}\left(\phi\right):=\left\{(x,t)\in \mathcal{H}\times\mathbb{R}\mid \phi(x)\leq t\right\}.
$$
\end{myDefn}
Recall that $\phi \,: \, \mathcal{H}\rightarrow \mathbb{R}\cup\left\{\pm \infty \right\}$ is said to be \textit{proper} if $\mathrm{dom}(\phi)\neq \emptyset$ and $\phi(x)>-\infty$ 
for all~$x\in\mathcal{H}$. Moreover, $\phi$ is a convex (resp.\ lower semi-continuous) function on $\mathcal{H}$ if and only if~$\mathrm{epi}(\phi)$ is a convex (resp.\ closed) subset of~$\mathcal{H}\times\R$.

\begin{myDefn}[Convex subdifferential operator]
 Let $\phi  :  \mathcal{H} \rightarrow \R\cup\left\{+\infty\right\}$ be a proper function. We denote by $\partial{\phi}  :  \mathcal{H} \rightrightarrows \mathcal{H}$ the \textit{convex subdifferential operator} of $\phi$, defined by 
 $$
 \partial{\phi}(x):=\left\{y\in\mathcal{H} \mid \forall z\in\mathcal{H}\text{, } \dual{y}{z-x}_{\mathcal{H}}\leq \phi(z)-\phi(x)\right\},
 $$
for all $x\in\mathcal{H}$.
\end{myDefn}

\begin{myDefn}[Proximal operator]\label{proxi}
 Let $\phi  :  \mathcal{H} \rightarrow \R\cup\left\{+\infty\right\}$ be a proper lower semi-continuous convex function. The \textit{proximal operator} associated with $\phi$ is the map~$\mathrm{prox}_{\phi}  :  \mathcal{H} \rightarrow \mathcal{H}$ defined by
 $$
      \mathrm{prox}_{\phi}(x):=\underset{y\in \mathcal{H}}{\argmin}\left[ \phi(y)+\frac{1}{2}\left \| y-x \right \|^{2}_{\mathcal{H}}\right]=(\mathrm{I}+\partial \phi)^{-1}(x),
 $$
for all $x\in\mathcal{H}$, where $\mathrm{I}  :  \mathcal{H}\rightarrow \mathcal{H}$ stands for the identity operator.
\end{myDefn}

It is well-known that, if $\phi  :  \mathcal{H} \rightarrow \R\cup\left\{+\infty\right\}$ is a proper lower semi-continuous convex function, then~$\partial{\phi}$ is a maximal monotone operator (see, e.g.,~\cite{ROCK2}), and thus the proximal operator~$\mathrm{prox}_{\phi}$ is well-defined, single-valued and nonexpansive, i.e. Lipschitz continuous with modulus $1$ (see, e.g.,~\cite[Chapter II]{BREZ2}).

\begin{myRem}\label{projetprox}\normalfont
    Note that, if $\phi:=\iota_{\mathrm{C}}$, where $\iota_{\mathrm{C}}$ is the indicator function of a nonempty closed convex subset ${\mathrm{C}}\subset\mathcal{H}$ , then $\iota_{\mathrm{C}}$ is a proper lower semi-continuous convex function and
    $$
\mathrm{prox}_{\iota_{\mathrm{C}}}=\mathrm{proj}_{\mathrm{C}},
    $$
    where $\mathrm{proj}_{\mathrm{C}}$ is the projection operator on $\mathrm{C}$.
\end{myRem}

\medskip

In the introduction, since we expressed the unique solution to the Signorini problem using the proximal operator, then the shape sensitivity analysis of this problem is related to the differentiability (in a particular sense) of the involved proximal operator. This issue is investigated using the notion of twice epi-differentiability~(see~\cite{Rockafellar}) defined as the Mosco epi-convergence of second-order difference quotient functions. In what follows we provide reminders and backgrounds on these notions and, for more details, we refer to~\cite[Chapter 7, Section B p.240]{ROCK} for the finite-dimensional case and to~\cite{DO} for the infinite-dimensional case. The strong (resp.\ weak) convergence of a sequence in~$\mathcal{H}$ will be denoted by~$\rightarrow$ (resp.\ $\rightharpoonup$) and note that all limits with respect to~$t$ will be considered for~$t \to 0^+$.
\begin{myDefn}[Mosco-convergence]
The \textit{outer}, \textit{weak-outer}, \textit{inner} and \textit{weak-inner limits} of a parameterized family~$(A_{t})_{t>0}$ of subsets of $\mathcal{H}$ are respectively defined by
\begin{eqnarray*}
      \mathrm{lim}\sup A_{t}&:=&\left\{ x\in \mathcal{H} \mid \exists (t_{n})_{n\in\mathbb{N}}\rightarrow 0^{+}, \exists \left(x_{n}\right)_{n\in\mathbb{N}}\rightarrow x, \forall n\in\mathbb{N}, x_{n}\in A_{t_{n}}\right\},\\
     \mathrm{w}\text{-}\mathrm{lim}\sup A_{t}&:=&\left\{ x\in \mathcal{H} \mid \exists (t_{n})_{n\in\mathbb{N}}\rightarrow 0^{+}, \exists \left(x_{n}\right)_{n\in\mathbb{N}}\rightharpoonup x, \forall n\in\mathbb{N}, x_{n}\in A_{t_{n}}\right\},\\
     \mathrm{lim}\inf A_{t}&:=&\left\{ x\in \mathcal{H} \mid \forall (t_{n})_{n\in\mathbb{N}}\rightarrow 0^{+}, \exists \left(x_{n}\right)_{n\in\mathbb{N}}\rightarrow x, \exists N\in\mathbb{N}, \forall n\geq N, x_{n}\in A_{t_{n}}\right\},\\
     \mathrm{w}\text{-}\mathrm{lim}\inf A_{t}&:=&\left\{ x\in \mathcal{H} \mid \forall (t_{n})_{n\in\mathbb{N}}\rightarrow 0^{+}, \exists \left(x_{n}\right)_{n\in\mathbb{N}}\rightharpoonup x, \exists N\in\mathbb{N}, \forall n\geq N, x_{n}\in A_{t_{n}}\right\}.
\end{eqnarray*}
The family~$(A_{t})_{t>0}$ is said to be \textit{Mosco-convergent} if~$
\mathrm{w}\text{-}\mathrm{lim}\sup A_{t}\subset\mathrm{lim}\inf A_{t}
$. In that case, all the previous limits are equal and we write
$$
     \mathrm{M}\text{-}\mathrm{lim} A_{t}:=\mathrm{lim}\inf A_{t}=\mathrm{lim}\sup A_{t}=\mathrm{w}\text{-}\mathrm{lim}\inf A_{t}=\mathrm{w}\text{-}\mathrm{lim}\sup A_{t}.
$$
\end{myDefn}
\begin{myDefn}[Mosco epi-convergence]
  Let $(\phi_{t})_{t>0}$ be a parameterized family of functions~$\phi_{t}  : \mathcal{H}\rightarrow \mathbb{R}\cup\left\{\pm \infty \right\}$ for all $t>0$.
 We say that $(\phi_{t})_{t>0}$ is \textit{Mosco epi-convergent} if~$(\mathrm{epi}(\phi_{t}))_{t>0}$ is Mosco-convergent in~$\mathcal{H} \times \R$. Then we denote by $\mathrm{ME}\text{-}\mathrm{lim}~ \phi_{t}  :  \mathcal{H}\rightarrow \mathbb{R}\cup\left\{\pm \infty \right\}$ the function characterized by its epigraph~$\mathrm{epi}\left(\mathrm{ME}\text{-}\mathrm{lim}~\phi_{t}\right):=\mathrm{M}\text{-}\mathrm{lim}$ $\displaystyle \mathrm{epi}\left(\phi_{t}\right)$ and we say that $(\phi_{t})_{t>0}$ Mosco epi-converges to~$\mathrm{ME}\text{-}\mathrm{lim}~\phi_{t}$.
 \end{myDefn}
 
Now let us recall the notion of twice epi-differentiability introduced by R.T.~Rockafellar in~1985 (see~\cite{Rockafellar}) that generalizes the classical notion of second-order derivative to nonsmooth convex functions.
 \begin{myDefn}[Twice epi-differentiability]\label{epidiff}
   A proper lower semi-continuous convex function~$\phi  : \mathcal{H}\rightarrow \mathbb{R}\cup\left\{+\infty \right\}$ is said to be \textit{twice epi-differentiable} at $x\in\mathrm{dom}(\phi)$ for $y\in\partial\phi(x)$ if the family of second-order difference quotient functions $(\delta_{t}^{2}\phi(x|y))_{t>0}$ defined by
$$
  \fonction{ \delta_{t}^{2}\phi(x|y) }{\mathcal{H}}{\mathbb{R}\cup\left\{+\infty\right\}}{z}{\displaystyle\frac{\phi(x+t z)-\phi(x)-t\dual{ y}{z}_{\mathcal{H}}}{\frac{1}{2}t^{2}},}
$$
for all $t>0$, is Mosco epi-convergent. In that case we denote by
$$
\mathrm{d}_{e}^{2}\phi(x|y):=\mathrm{ME}\text{-}\mathrm{lim}~\delta_{t}^{2}\phi(x|y),
$$
which is called the second-order epi-derivative of $\phi$ at $x$ for $y$.
\end{myDefn}

In this paper we have to deal with the twice epi-differentiability of the Signorini indicator functional. For this purpose, we use a result, due to C.N. Do. (see~\cite[Chapter 2, Example 2.10 p.287]{DO}), which shows that the indicator function of a nonempty closed convex set is twice epi-differentiable under an appropriate assumption. Before to introduce this result, let us recall some classical definitions from convex analysis.

\begin{myDefn}[Normal cone]\label{conenormall}
Let $\mathrm{C}$ be a nonempty closed convex subset of $\mathcal{H}$ and $x\in\mathrm{C}$. The \textit{normal cone} to $\mathrm{C}$ at $x$ is the nonempty closed convex cone of $\mathcal{H}$ defined by
$$
\mathrm{N}_{\mathrm{C}}(x):=\left\{z\in\mathcal{H}\mid \dual{z}{c-x}_{\mathcal{H}}\leq0, \forall c\in\mathrm{C}\right\}.
$$
\end{myDefn}

\begin{myExa}\label{ExamNormal}
    Let $\mathrm{C}$ be a nonempty closed convex subset of $\mathcal{H}$ and $\iota_{\mathrm{C}}$ be the indicator function of~$\mathrm{C}$. Then, for all $x\in\mathrm{C}$,
    $$
    \partial{\iota}_{\mathrm{C}} (x)=\mathrm{N}_{\mathrm{C}}(x).
    $$
\end{myExa}

\begin{myDefn}[Tangent cone]\label{tangentcone}
Let $\mathrm{C}$ be a nonempty closed convex subset of $\mathcal{H}$ and $x\in\mathrm{C}$. The tangent cone to $\mathrm{C}$ at $x$ is the nonempty closed convex cone of $\mathcal{H}$ defined by
$$
\mathrm{T}_{\mathrm{C}}(x):=\overline{\left\{z\in\mathcal{H}\mid \exists \lambda>0, x+\lambda z\in\mathrm{C}\right\}}.
$$
\end{myDefn}

\begin{myDefn}[Polyhedric set]\label{polyset}
    Let $\mathrm{C}$ be a nonempty closed convex subset of $\mathcal{H}$. We say that~$\mathrm{C}$ is polyhedric at $x\in\mathrm{C}$ for $y\in\mathrm{N}_{\mathrm{C}}(x)$ if 
    $$
    \mathrm{T}_{\mathrm{C}}(x)\cap\left(\R y\right)^{\perp}=\overline{\left\{z\in\mathcal{H}\mid \exists \lambda>0, x+\lambda z\in\mathrm{C}\right\}\cap\left(\R y\right)^{\perp}}.
    $$
\end{myDefn}

\begin{myRem}\normalfont
    Recall that, in finite dimension, polyhedric sets reduce to polyhedral sets, which is the intersection of a finite set of closed half-spaces (see, e.g.,~\cite{MANA}).
\end{myRem}

\begin{myLem}\label{epipoly}
Let $\mathrm{C}$ be a nonempty closed convex subset of $\mathcal{H}$ and $\iota_{\mathrm{C}}$ be the indicator function of~$\mathrm{C}$. If $\mathrm{C}$ is polyhedric at $x\in\mathrm{C}$ for~$y\in\mathrm{N}_{\mathrm{C}}(x)$, then $\iota_\mathrm{C}$ is twice epi-differentiable at $x$ for $y$ and
$$
\mathrm{d}_{e}^{2}\iota_\mathrm{C}(x|y)=\mathrm{\iota}_{\mathrm{T}_{\mathrm{C}}(x)\cap\left(\R y\right)^{\perp}}.
$$
\end{myLem}

Let us conclude this section with a last proposition (see, e.g.,~\cite{ROCKAGENE,ROCK} for the finite-dimensional
case and~\cite{8AB,DO} for the infinite-dimensional one). We bring to the attention of the reader that Proposition~\ref{TheoABC2018} is the key point in order to derive our main results.

\begin{myProp}\label{TheoABC2018}
Let~$\Phi : \mathcal{H}\rightarrow \mathbb{R} \cup \left\{+\infty\right\}$ be a proper lower semi-continuous convex function on~$\mathcal{H}$. Let $F  :  \mathbb{R}^{+}\rightarrow \mathcal{H}$ and let~$u  :  \mathbb{R}^{+}\rightarrow \mathcal{H}$ be defined by
$$
    u(t):=\mathrm{prox}_{\Phi}(F(t)),
$$
for all~$t\geq 0$. If the conditions 
\begin{enumerate}[label=\arabic*)]
    \item $F$ is differentiable at $t=0$;
    \item $\Phi$ is twice epi-differentiable at $u(0)$ for $F(0)-u(0)\in\partial \Phi(u(0))$;
\end{enumerate}
are both satisfied, then $u$ is differentiable at $t=0$ with
$$
u'(0)=\mathrm{prox}_{\mathrm{d}_{e}^{2}\Phi(u(0)|F(0)-u(0))}(F'(0)).
$$
\end{myProp}

\subsection{Functional framework and some required boundary value problems}\label{BVP}

The major part of the present work consists in performing the sensitivity analysis of a Signorini problem with respect to shape perturbation. For this purpose, we present in this section the functional framework and some boundary value problems: a Dirichlet-Neumann problem and a Signorini problem. We present some basic notions and results concerning these boundary value problems for the reader's convenience. Since the proofs are very similar to the ones detailed in the paper~\cite{4ABC}, they will be omitted here.

Let $d\in\left\{2,3\right\}$ and $\Omega$ be a nonempty bounded connected open subset of~$\R^{d}$ with a Lipschitz-boundary $\Gamma :=\partial{\Omega}$ and $\nn$ the outward-pointing unit normal vector to $\Gamma$. In the whole paper we denote by $\MD(\Omega,\R^d)$ the set of functions that are infinitely differentiable with
compact support in~$\Omega$, by $\MD'(\Omega,\R^d)$ the set of distributions on $\Omega$, for $(m,p) \in \mathbb{N}\times\mathbb{N}^*$, by~$\mathrm{W}^{m,p}(\Omega,\R^d)$, $\LL^{2}(\Gamma,\R^d)$, $\HH^{1/2}(\Gamma,\R^d)$, $\HH^{-1/2}(\Gamma,\R^d)$, the usual Lebesgue and Sobolev spaces endowed with their standard norms, and we denote by $\HH^{m}(\Omega,\R^d):=\mathrm{W}^{m,2}(\Omega,\R^d)$ and by~$\HH_{\mathrm{div}}(\Omega, \R^{d\times d}):= \{ w\in \mathrm{L}^{2}(\Omega,\R^{d\times d})  \mid \mathrm{div} (w)\in \mathrm{L}^{2}(\Omega,\R^d) \}$, where $\mathrm{div}(w)$ is the vector whose the $i$-th component is defined by $\mathrm{div}(w)_i:=\mathrm{div}(w_{i})\in\LL^2(\Omega,\R)$, and where $w_i\in\mathrm{L}^{2}(\Omega,\R^{d})$ is the transpose of the~$i$-th line of~$w$, for all $i\in[[1,d]]$ and for all~$w\in\HH_{\mathrm{div}}(\Omega, \R^{d\times d})$. Moreover, for any $w\in \HH_{\mathrm{div}}(\Omega, \R^{d\times d})$, we denote by~$w\nn$ the function in $\HH^{-1/2}(\Gamma,\R^d)$ given by the divergence formula (see Proposition~\ref{div}). All along this paper we denote by~$:$  the scalar product defined by $\mathrm{B}:\mathrm{C}=\sum_{i=1}^{d}\mathrm{B}_i\cdot\mathrm{C}_i$, for all $\mathrm{B},\mathrm{C}\in\R^{d\times d}$, where $\mathrm{B}_i\in\R^d$ (resp. $\mathrm{C}_i\in\R^d$) stands for the transpose of the~$i$-th line of~$\mathrm{B}$ (resp. $\mathrm{C}$) for all $i\in[[1,d]]$.

Let us assume that~$\Gamma$ is decomposed as~$\Gamma_{\mathrm{D}}\cup\Gamma_{\mathrm{S}}$, where~$\Gamma_{\mathrm{D}}$ and $\Gamma_{\mathrm{S}}$  are two measurable pairwise disjoint subsets of $\Gamma$, such that $\Gamma_{\mathrm{D}}$ has a positive measure. In that case we denote $\HH^{1}_{\mathrm{D}}(\Omega,\R^d)$ the linear subspace of $\HH^{1}(\Omega,\R^d)$ defined by
$$
\HH^{1}_{\mathrm{D}}(\Omega,\R^d):=\left\{v\in\HH^{1}(\Omega,\R^d)\mid v=0 \text{ \textit{a.e.} on } \Gamma_{\mathrm{D}} \right\}.
$$

In the linear elastic model (see, e.g.,~\cite{SALEN}) the Cauchy stress tensor denoted $\sigma$ is given by
$$
\sigma(v)=\mathrm{A}\mathrm{e}(v),
$$
where $\mathrm{A}$ the stiffness tensor, and $\mathrm{e}$ is the infinitesimal strain tensor defined by 
    \begin{equation*}
    \mathrm{e}(v):=\frac{1}{2}(\nabla{v}+\nabla{ v}^{\top}),
    \end{equation*}
for all displacement field $v\in\HH^1(\Omega,\R^d)$. We also assume that all coefficients of $\mathrm{A}$ are constant (denoted by $a_{ijkl}$ for all $\left(i,j,k,l\right)\in\left\{1,...,d\right\}^{4}$), and there exists one constant $\alpha>0$ such that all coefficients of $\mathrm{A}$ and $e$ (denoted by $\epsilon_{ij}$ for all $\left(i,j\right)\in\left\{1,...,d\right\}^{2}$) satisfy
 \begin{equation*}
a_{ijkl}=a_{jikl}=a_{lkij}\text{, }
\displaystyle\sum_{i=1}^{d}\sum_{j=1}^{d}\sum_{k=1}^{d}\sum_{l=1}^{d}a_{ijkl}\epsilon_{ij}(v_1)(x)\epsilon_{kl}(v_2)(x)\geq\alpha\sum_{i=1}^{d}\sum_{j=1}^{d}\epsilon_{ij}(v_1)(x)\epsilon_{ij}(v_2)(x),
\end{equation*}    
for all displacement field $v_1,v_2\in\HH^1(\Omega,\R^d)$ and for almost all $x\in\Omega$. From the symmetry assumption on $\mathrm{A}$, note that $\mathrm{A}\mathrm{e}(v)=\mathrm{A}\nabla{v}$, for all~$v\in\HH^1_{\mathrm{D}}(\Omega,\R^d)$. Moreover, since $\Gamma_{\mathrm{D}}$ has a positive measure, then 
$$
\fonction{\dual{\cdot}{\cdot}_{\HH^{1}_{\mathrm{D}}(\Omega,\R^{d})}}{\left(\HH^{1}_{\mathrm{D}}(\Omega,\R^{d})\right)^2}{\R}{(v_1,v_2)}{\displaystyle\int_{\Omega}\mathrm{A}\mathrm{e}(v_1):\mathrm{e}(v_2),}
$$
is a scalar product on $\HH^{1}_{\mathrm{D}}(\Omega,\R^{d})$ (see, e.g.,~\cite[Chapter 3]{DUVAUTLIONS}), and we denote by $\left\|\cdot\right\|_{\HH^{1}_{\mathrm{D}}(\Omega,\R^{d})}$ the corresponding norm.

For any $v\in\LL^{2}(\Gamma,\R^d)$, one writes
$v=v_{\nn}\nn+v_\tau$,
where $v_{\nn}:=v\cdot\nn\in\LL^{2}(\Gamma,\R)$ 
and $v_{\tau}:=v-v_{\nn}\nn\in\LL^2(\Gamma,\R^d)$.
In particular, if the stress vector $\mathrm{A}\mathrm{e}(v)\nn$ belongs to $\LL^2(\Gamma_{\mathrm{S}},\R^d)$ for some $v\in\HH^1(\Omega,\R^d)$, then we use the notation 
$$
\mathrm{A}\mathrm{e}(v)\nn=\sigma_{\nn}(v)\nn+\sigma_\tau(v),
$$
where $\sigma_{\nn}(v)\in\LL^{2}(\Gamma_{\mathrm{S}},\R)$ is the normal stress and $\sigma_{\tau}\in\LL^{2}(\Gamma_{\mathrm{S}},\R^d)$ the shear stress. We also denote, for all $(x,y)\in\R^d\times\R^d$, by $xy^{\top}$ the matrix whose the $i$-th line is given by the vector $x_iy$, where~$x_i\in\R$ is the $i$-th component of $x$, for all~$i\in[[1,d]]$.

In the sequel, consider $k\in\LL^2(\Omega,\R^d)$, $h\in\LL^2(\Gamma_{\mathrm{S}},\R^d)$ and $w\in\HH^1_{\mathrm{D}}(\Omega,\R^d)$.

\subsubsection{A Problem with Dirichlet-Neumann Conditions}
Consider the Dirichlet-Neumann problem given by
\begin{equation}\tag{DN}\label{PbNeumannDirichlet}
\arraycolsep=2pt
\left\{
\begin{array}{rcll}
-\mathrm{div}(\mathrm{A}\mathrm{e}(F)) & = & k   & \text{ in } \Omega , \\
F & = & 0  & \text{ on } \Gamma_{\mathrm{D}} ,\\
\mathrm{A}\mathrm{e}(F)\nn & = & h  & \text{ on } \Gamma_{\mathrm{S}}.
\end{array}
\right.
\end{equation}

\begin{myDefn}[Strong solution to the Dirichlet-Neumann problem]
A (strong) solution to the Dirichlet-Neumann problem~\eqref{PbNeumannDirichlet} is a function $F\in\HH^1(\Omega,\R^{d})$ such that $-\mathrm{div}(\mathrm{A}\mathrm{e}(F))=k$ in~$\MD'(\Omega,\R^d)$,~$F=0$ \textit{a.e.} on $\Gamma_{\mathrm{D}}$, $\mathrm{A}\mathrm{e}(F)\nn\in\LL^{2}(\Gamma_{\mathrm{S}},\R^d)$ with $\mathrm{A}\mathrm{e}(F)\nn=h$ \textit{a.e.} on $\Gamma_{\mathrm{S}}$.
\end{myDefn}

\begin{myDefn}[Weak solution to the Dirichlet-Neumann problem]
A weak solution to the Dirichlet-Neumann problem~\eqref{PbNeumannDirichlet} is a function $F\in\HH^{1}_{\mathrm{D}}(\Omega,\R^{d})$ such that
\begin{equation*}
\int_{\Omega}\mathrm{A}\mathrm{e}(F):\mathrm{e}(v)=\int_{\Omega}k\cdot v+\int_{\Gamma_{\mathrm{S}}}h\cdot v, \qquad \forall v\in\HH^{1}_{\mathrm{D}}(\Omega,\R^{d}).
\end{equation*}
\end{myDefn}

\begin{myProp}
A function $F\in\HH^1(\Omega,\R^{d})$ is a (strong) solution to the Dirichlet-Neumann problem~\eqref{PbNeumannDirichlet} if and only if $F$ is a weak solution to the Dirichlet-Neumann problem~\eqref{PbNeumannDirichlet}.
\end{myProp}

Using the Riesz representation theorem, we obtain the following existence/uniqueness result.

\begin{myProp}
The Dirichlet-Neumann problem~\eqref{PbNeumannDirichlet} admits a unique solution $F\in\HH^{1}_{\mathrm{D}}(\Omega,\R^{d})$. 
\end{myProp}

\subsubsection{A Signorini problem}\label{SectionSignorinicasscalairesansu}

In this part, let us assume that $\Gamma_{\mathrm{S}}$ is decomposed, up to a null set, as 
$
\Gamma_{\mathrm{S}_{\mathrm{N}}}\cup\Gamma_{\mathrm{S}_{\mathrm{D}}}\cup\Gamma_{\mathrm{S}_{\mathrm{S}}},
$
where~$\Gamma_{\mathrm{S}_{\mathrm{N}}}$,~$\Gamma_{\mathrm{S}_{\mathrm{D}}}$ and $\Gamma_{\mathrm{S}_{\mathrm{S}}}$ are three measurable pairwise disjoint subsets of $\Gamma_{\mathrm{S}}$.
Consider the Signorini problem given by
\begin{equation}\tag{SiP}\label{PbtangSignorini}
\arraycolsep=2pt
\left\{
\begin{array}{rcll}
-\mathrm{div}(\mathrm{A}\mathrm{e}(u)) & = & k   & \text{ in } \Omega , \\
u & = & 0  & \text{ on } \Gamma_{\mathrm{D}} ,\\
\sigma_{\tau}(u) & = & h_\tau  & \text{ on } \Gamma_{\mathrm{S}} ,\\
\sigma_{\nn}(u) & = & h_{\nn}   & \text{ on } \Gamma_{\mathrm{S}_{\mathrm{N}}} ,\\
u_{\nn} & = & w_{\nn}  & \text{ on } \Gamma_{\mathrm{S}_{\mathrm{D}}} ,\\
u_{\nn}\leq w_{\nn}\text{, } \sigma_{\nn}(u)\leq h_{\nn} \text{ and } \left(u_{\nn}-w_{\nn}\right)\left(\sigma_{\nn}(u)-h_{\nn} \right)& = & 0  & \text{ on } \Gamma_{\mathrm{S}_{\mathrm{S}}},
\end{array}
\right.
\end{equation}
where the data are given at the beginning of Section~\ref{BVP}.
\begin{myDefn}[Strong solution]
A (strong) solution to the problem~\eqref{PbtangSignorini} is a function $u\in\HH^1(\Omega,\R^{d})$ such that $-\mathrm{div}(\mathrm{A}\mathrm{e}(u))=k$ in $\MD'(\Omega,\R^d)$,~$u=0$ \textit{a.e.} on $\Gamma_{\mathrm{D}}$, $u_{\nn}=w_{\nn}$ \textit{a.e.} on $\Gamma_{\mathrm{S}_{\mathrm{D}}} $, $\mathrm{A}\mathrm{e}(u)\nn\in\LL^{2}(\Gamma_{\mathrm{S}},\R^d)$ with $\sigma_{\tau}(u)=h_\tau$ \textit{a.e.} on $\Gamma_{\mathrm{S}}$, $\sigma_{\nn}=h_{\nn}$ \textit{a.e.} on $\Gamma_{\mathrm{S}_{\mathrm{N}}}$, $u_{\nn}\leq w_{\nn}$, $\sigma_{\nn}(u)\leq h_{\nn} \text{ and } (u_{\nn}-w_{\nn})(\sigma_{\nn}(u)-h_{\nn})=0$ \textit{a.e.} on $\Gamma_{\mathrm{S}_{\mathrm{S}}}$.
\end{myDefn}

\begin{myDefn}[Weak solution]
A weak solution to problem~\eqref{PbtangSignorini} is a function $u\in\mathcal{K}^{1}_w(\Omega)$ such that
\begin{equation*}
\displaystyle\int_{\Omega}\mathrm{A}\mathrm{e}(u):\mathrm{e}(v-u)\geq\int_{\Omega}k\cdot (v-u)+\int_{\Gamma_{\mathrm{S}}}h\cdot \left(v-u\right), \qquad \forall v\in\mathcal{K}^{1}_w(\Omega),
\end{equation*}
where $\mathcal{K}^{1}_w(\Omega)$ is the nonempty closed convex subset of $\HH^{1}_{\mathrm{D}}(\Omega,\R^{d})$ defined by
$$
\mathcal{K}^{1}_w(\Omega) := \left\{v\in\HH^{1}_{\mathrm{D}}(\Omega,\R^{d}) \mid v_{\nn}=w_{\nn} \text{ \textit{a.e.} on } \Gamma_{\mathrm{S}_{\mathrm{D}}}\text{ and }v_{\nn}\leq w_{\nn}\text{ \textit{a.e.} on }\Gamma_{\mathrm{S}_{\mathrm{S}}}  \right \}.
$$
\end{myDefn}

One can prove that a (strong) solution is a weak solution but, to the best of our knowledge, without additional assumption, one cannot prove the converse. To get the equivalence, we need to assume, in particular, that the decomposition $\Gamma_{\mathrm{D}}\cup\Gamma_{\mathrm{S}_{\mathrm{N}}}\cup\Gamma_{\mathrm{S}_{\mathrm{D}}}\cup\Gamma_{\mathrm{S}_{\mathrm{S}}}$ of $\Gamma$ is \textit{consistent} in the following sense.
\begin{myDefn}[Consistent decomposition]\label{regulieresens2}
 The decomposition   $\Gamma_{\mathrm{D}}\cup\Gamma_{\mathrm{S}_{\mathrm{N}}}\cup\Gamma_{\mathrm{S}_{\mathrm{D}}}\cup\Gamma_{\mathrm{S}_{\mathrm{S}}}$ of $\Gamma$ is said to be \textit{consistent} if
 \begin{enumerate}
     \item for almost all $s\in\Gamma_{\mathrm{S}_{\mathrm{S}}}$,  $s\in \mathrm{int}_{\Gamma}(\Gamma_{\mathrm{S}_{\mathrm{S}}})$;
     \item the nonempty closed convex subset $\mathcal{K}^{1/2}_{w}(\Gamma)$ of $\HH^{1/2}(\Gamma,\R^d)$ defined by
    \begin{multline*}
            \mathcal{K}^{1/2}_{w}(\Gamma):=\biggl\{ v\in \HH^{1/2}(\Gamma,\R^d) \mid v=0 \text{ \textit{a.e.} on } \Gamma_{\mathrm{D}}\text{, } v_{\nn}=w_{\nn} \text{ \textit{a.e.} on } \Gamma_{\mathrm{S}_{\mathrm{D}}} \\ \text{ and }v_{\nn}\leq w_{\nn}\text{ \textit{a.e.} on }\Gamma_{\mathrm{S}_{\mathrm{S}}}  \biggl\},
    \end{multline*}        
is dense in  the nonempty closed convex subset $\mathcal{K}^{0}_{w}(\Gamma)$ of $\mathrm{L}^{2}(\Gamma,\R^d)$ defined by
     \begin{multline*}
         \mathcal{K}^{0}_{w}(\Gamma):=\biggl\{ v\in \mathrm{L}^{2}(\Gamma,\R^d) \mid v=0 \text{ \textit{a.e.} on } \Gamma_{\mathrm{D}}\text{, }v_{\nn}=w_{\nn} \text{ \textit{a.e.} on } \Gamma_{\mathrm{S}_{\mathrm{D}}} \\ \text{ and }v_{\nn}\leq w_{\nn}\text{ \textit{a.e.} on }\Gamma_{\mathrm{S}_{\mathrm{S}}} \biggl \}.
    \end{multline*}

\end{enumerate}
\end{myDefn}

\begin{myProp}\label{EquiSignorini}
Let $u\in \HH^1(\Omega,\R^{d})$.
\begin{enumerate}
    \item If $u$ is a (strong) solution to the problem~\eqref{PbtangSignorini}, then $u$ is a weak solution to the problem~\eqref{PbtangSignorini}.
    \item If $u$ is a weak solution to the problem~\eqref{PbtangSignorini} such that $\mathrm{A}\mathrm{e}(u)\nn\in\LL^{2}(\Gamma_{\mathrm{S}},\R^d)$ and the decomposition $\Gamma_{\mathrm{D}}\cup\Gamma_{\mathrm{S}_{\mathrm{N}}}\cup\Gamma_{\mathrm{S}_{\mathrm{D}}}\cup\Gamma_{\mathrm{S}_{\mathrm{S}}}$ of $\Gamma$ is consistent, then $u$ is a (strong) solution to the problem~\eqref{PbtangSignorini}.
\end{enumerate}
\end{myProp}

\begin{myProp}
The problem~\eqref{PbtangSignorini} admits a unique weak solution $u\in\HH^{1}_{\mathrm{D}}(\Omega,\R^d)$ which is given by
$$
    \displaystyle u=\mathrm{prox}_{\iota_{\mathcal{K}^{1}_w(\Omega)}}(F),
$$
where $F\in\HH^{1}_{\mathrm{D}}(\Omega,\R^{d})$ is the unique solution to the Dirichlet-Neumann problem~\eqref{PbNeumannDirichlet}, and $\mathrm{prox}_{\iota_{\mathcal{K}^{1}_w(\Omega)}}$ stands for the proximal operator associated with the indicator function $\iota_{\mathcal{K}^{1}_w(\Omega)}$ considered on the Hilbert space $(\HH^{1}_{\mathrm{D}}(\Omega,\R^{d}),\dual{\cdot}{\cdot}_{\HH^{1}_{\mathrm{D}}(\Omega,\R^{d})})$.
\end{myProp}

\begin{myRem}\normalfont
    Note that, from Remark~\ref{projetprox}, the unique weak solution $u\in\HH^{1}_{\mathrm{D}}(\Omega,\R^{d})$ to the problem~\eqref{PbtangSignorini} is also characterized by the projection operator since $\mathrm{prox}_{\iota_{\mathcal{K}^{1}_w(\Omega)}}= \mathrm{proj}_{\mathcal{K}^{1}_w(\Omega)}$.
\end{myRem}

\section{Main results}\label{mainresultoff}

Let $d\in\left\{2,3\right\}$ and~$f\in\HH^1(\R^d,\R^d)$. Let $\Omega_{\mathrm{ref}}$ be a nonempty connected bounded open subset of~$\R^d$ with Lipschitz boundary $\Gamma_{\mathrm{ref}}:=\partial{\Omega_{\mathrm{ref}}}$. We assume that $\Gamma_{\mathrm{ref}}=\Gamma_{\mathrm{D}}\cup{\Gamma_{\mathrm{S}_{\mathrm{ref}}}}$, where $\Gamma_{\mathrm{D}}$ and~${\Gamma_{\mathrm{S}_{\mathrm{ref}}}}$ are two measurable pairwise disjoint subsets of $\Gamma_{\mathrm{ref}}$, such that $\Gamma_{\mathrm{D}}$ has a positive measure. We consider the set of admissible shapes~$\mathcal{U}_{\mathrm{ref}}$ defined in~\eqref{setadmiss}.
Note that, all shapes in $\mathcal{U}_{\mathrm{ref}}$ have $\Gamma_{\mathrm{D}}$ as common boundary part. 

We consider the shape optimization problem~\eqref{shapeOptim}. From Subsection~\ref{SectionSignorinicasscalairesansu}, note that the Signorini energy functional $\mathcal{J}$, given by~\eqref{energysigno}, can also be expressed as
$$
  \mathcal{J}(\Omega) = -\frac{1}{2}\int_{\Omega}\mathrm{A}\mathrm{e}\left(u_\Omega\right):\mathrm{e}\left(u_\Omega\right),
$$
for all~$\Omega \in \mathcal{U}_{\mathrm{ref}}$.

In the whole section let us fix~$\Omega_0 \in \mathcal{U}_{\mathrm{ref}}$. Our aim here is to prove that, under appropriate assumptions, the functional $\mathcal{J}$ is \textit{shape differentiable} at $\Omega_0$, in the sense that the map
$$
\begin{array}{rcl}
\mathcal{C}_{\mathrm{D}}^{2,\infty}(\R^{d},\R^{d})& \longrightarrow & \R \\
\theta & \longmapsto & \displaystyle \mathcal{J}((\mathrm{id}+\theta)(\Omega_0)),
\end{array}
$$
is Gateaux differentiable at~$0$, where $\mathcal{C}_{\mathrm{D}}^{2,\infty}(\R^{d},\R^{d})$ is defined by~\eqref{direc12}, and to give an expression of the Gateaux differential, denoted by~$\mathcal{J}'(\Omega_0)$, which is called the shape gradient of~$\mathcal{J}$ at~$\Omega_0$. For this purpose we have to perform the sensitivity analysis of the Signorini problem~\eqref{Signoriniproblem2221} with respect to the shape, and then to characterize the material and shape derivatives. 

This section is separated as follows. In Subsection~\ref{setting}, we perturb the Signorini problem with respect to the shape. In Subsection~\ref{sectionepidiff} we characterize the material derivative as solution to a variational inequality (see Theorem~\ref{materialderiv1}). Then, with additional regularity assumptions, we characterize the material and shape derivatives as being weak solutions to Signorini problems (see Corollaries~\ref{materialderiv2} and~\ref{shapederiv1}). Finally, in Subsection~\ref{energyfunctional}, we prove that the Signorini energy functional~$\mathcal{J}$ is shape differentiable at $\Omega_0$, and we provide an expression of its shape gradient (see Theorem~\ref{shapederivofJsigno1} and Corollary~\ref{shapederivofJ}).

\subsection{Setting of the shape perturbation}\label{setting}
 Consider $\theta\in \mathcal{C}_{\mathrm{D}}^{2,\infty}(\R^{d},\R^{d})$ and, for all $t\geq0$ sufficiently small such that~$\mathrm{id}+t\theta$ is a $\mathcal{C}^{2}$-diffeomorphism of $\R^{d}$, consider the shape perturbed Signorini problem given by
\begin{equation}\tag{SP$_{t}$}\label{PbperturbSignorini}
\arraycolsep=2pt
\left\{
\begin{array}{rcll}
-\mathrm{div}(\mathrm{A}\mathrm{e}(u_{t})) & = & f   & \text{ in } \Omega_t , \\
u_{t} & = & 0  & \text{ on } \Gamma_{\mathrm{D}} ,\\
\sigma_{\tau_{t}}(u_{t}) & = & 0  & \text{ on } \Gamma_{\mathrm{S}_t} ,\\
u_{t,\nn_t}\leq0\text{, } \sigma_{\nn_t}(u_{t})\leq0 \text{ and } u_{t,\nn_t}\sigma_{\nn_t}(u_{t}) & = & 0  & \text{ on } \Gamma_{\mathrm{S}_t},
\end{array}
\right.
\end{equation}
where $\Omega_{t}:=(\mathrm{id}+t\theta)(\Omega_{0})\in\mathcal{U}_{\mathrm{ref}}$, $\Gamma_{t}:= (\mathrm{id}+t\theta)(\Gamma_{0})$ and $\nn_t$ is the outward-pointing unit normal vector to $\Gamma_t$. From Subsection~\ref{SectionSignorinicasscalairesansu}, there
exists a unique solution $u_{t}\in\HH^1(\Omega_t,\R^d)$ to~\eqref{PbperturbSignorini} which satisfies
\begin{equation*}\label{weakkk}
\displaystyle\int_{\Omega_t}\mathrm{A}\mathrm{e}(u_{t}):\mathrm{e}(v-u_{t})\geq\int_{\Omega_t}f\cdot (v-u_{t}), \qquad \forall v\in\mathcal{K}^{1}(\Omega_t),
\end{equation*}
where 
$$
\mathcal{K}^{1}(\Omega_t):=\left\{ v\in\HH^1_{\mathrm{D}}(\Omega_t,\R^d) \mid v_{\nn_{t}}\leq0\text{ \textit{a.e.} on } \Gamma_{\mathrm{S}_t} \right\}.
$$
Following the usual strategy in shape optimization literature (see, e.g.,~\cite{ALL,HENROT}), using the change of variables $\mathrm{id}+t\theta$ and the equality $$
\nn_t\circ(\mathrm{id}+t\theta)=\frac{(\mathrm{I}+t\nabla{\theta}^{\top})^{-1}\nn}{\left\|(\mathrm{I}+t\nabla{\theta}^{\top})^{-1}\nn\right\|},
$$ 
where $\nn:=\nn_0$ (see, e.g.,~\cite[Chapter 2, Proposition 2.48 p.79]{SOKOZOL}) and $||\cdot||$ is the Euclidean norm on~$\R^d$, we prove that $\overline{u}_t:=u_t\circ(\mathrm{id}+t\theta)\in\mathrm{K}^{1}_{t}(\Omega_0)\subset\HH^1_{\mathrm{D}}(\Omega_0,\R^d)$ satisfies
\begin{equation}\label{inequalitytogetderivmat}
\displaystyle\int_{\Omega_0}\mathrm{J}_{t}\mathrm{A}\left[\nabla{\overline{u}_t}\left(\mathrm{I}+t\nabla{\theta}\right)^{-1}\right]:\nabla{\left(v-\overline{u}_t\right)}\left(\mathrm{I}+t\nabla{\theta}\right)^{-1}\geq\int_{\Omega_0}f_{t}\mathrm{J}_{t}\cdot\left(v-\overline{u}_t\right), \qquad \forall v\in\mathrm{K}^{1}_{t}(\Omega_0),
\end{equation}
where $\mathrm{K}^{1}_{t}(\Omega_0):=\{ v\in\HH^1_{\mathrm{D}}(\Omega_0,\R^d) \mid v\cdot(\mathrm{I}+t\nabla{\theta}^{\top})^{-1}\nn\leq0 \text{ \textit{a.e.} on } \Gamma_{\mathrm{S}_0} \}$, $f_{t}:=f\circ(\mathrm{id}+t\theta)\in\HH^{1}(\R^{d},\R^d)$ and~$\mathrm{J}_{t}:=\mathrm{det}(\mathrm{I}+t\nabla{\theta})\in\LL^{\infty}(\R^{d},\R)$ is the Jacobian. Thus, using the characterization of the proximal operator (see Definition~\ref{proxi}),  $\overline{u}_t$ can be expressed as
$$
\overline{u}_t=\mathrm{prox}_{\iota_{\mathrm{K}^{1}_{t}(\Omega_0)}}(F_t),
$$
where $F_{t}\in\HH^{1}(\Omega_{0},\R^d)$ is the unique solution to the parameterized variational equality
\begin{equation*}
\displaystyle\int_{\Omega_0}\mathrm{J}_{t}\mathrm{A}\left[\nabla{F_t}\left(\mathrm{I}+t\nabla{\theta}\right)^{-1}\right]:\nabla{v}\left(\mathrm{I}+t\nabla{\theta}\right)^{-1}=\int_{\Omega_0}f_{t}\mathrm{J}_{t}\cdot v, \qquad \forall v\in\HH^1_{\mathrm{D}}(\Omega_0,\R^d),
\end{equation*}
and $\mathrm{prox}_{\iota_{\mathrm{K}^{1}_{t}(\Omega_0)}}$ is the proximal operator associated with the indicator function $\iota_{\mathrm{K}^{1}_{t}(\Omega_0)}$
considered on the space $\HH^1_{\mathrm{D}}(\Omega_0,\R^d)$ endowed with the perturbed scalar product
$$(v_1,v_2)\in\left(\HH^1_{\mathrm{D}}(\Omega_0,\R^d)\right)^2\mapsto \int_{\Omega_0}\mathrm{J}_{t}\mathrm{A}\left[\nabla{v_1}\left(\mathrm{I}+t\nabla{\theta}\right)^{-1}\right]:\nabla{v_2}\left(\mathrm{I}+t\nabla{\theta}\right)^{-1}\in\R.
$$

Here, the main difficulty is that the indicator function $\iota_{\mathrm{K}^{1}_{t}(\Omega_0)}$ depends on the parameter $t$, thus it would required an extended notion of twice epi-differentiability
depending on a parameter in order to apply the Proposition~\ref{TheoABC2018}. Nevertheless, it is not necessary since, for all $v\in\mathrm{K}^{1}_{t}(\Omega_0)$, one has (similarly to~\cite[Chapter 5 p.111]{MAURY} and~\cite[Chapter 4 Section 4.6 p.205]{SOKOZOL})
$$
\left(\mathrm{I}+t\nabla{\theta}\right)^{-1}v\in\mathcal{K}^1(\Omega_0),
$$
and conversely, for all $\varphi\in\mathcal{K}^1(\Omega_0)$,
$$
(\mathrm{I}+t\nabla{\theta})\varphi\in\mathrm{K}^{1}_{t}(\Omega_0).
$$
 Thus, from Inequality~\eqref{inequalitytogetderivmat}, one proves that $\overline{\overline{u}}_t:=(\mathrm{I}+t\nabla{\theta})^{-1}\overline{u}_{t}\in\mathcal{K}^{1}(\Omega_{0})$  satisfies
\begin{multline}\label{firstsi}
\displaystyle\int_{\Omega_0}\mathrm{J}_{t}\mathrm{A}\left[\nabla{\left(\left(\mathrm{I}+t\nabla{\theta}\right)\overline{\overline{u}}_t\right)}\left(\mathrm{I}+t\nabla{\theta}\right)^{-1}\right]:\nabla{\left(\left(\mathrm{I}+t\nabla{\theta}\right)\left(\varphi-\overline{\overline{u}}_t\right)\right)}\left(\mathrm{I}+t\nabla{\theta}\right)^{-1}\\\geq\int_{\Omega_0}\left(\mathrm{I}+t\nabla{\theta}^{\top}\right)f_{t}\mathrm{J}_{t}\cdot\left(\varphi-\overline{\overline{u}}_t\right), \qquad \forall \varphi\in\mathcal{K}^{1}(\Omega_0),
\end{multline}
and can be expressed as
$$
\overline{\overline{u}}_t=\mathrm{prox}_{\iota_{\mathcal{K}^{1}(\Omega_0)}}(G_t),
$$
where $G_t\in\HH^{1}_{\mathrm{D}}(\Omega_0,\R^d)$ is the unique solution to the parameterized variational equality
\begin{multline*}
\int_{\Omega_0}\mathrm{J}_{t}\mathrm{A}\left[\nabla{\left(\left(\mathrm{I}+t\nabla{\theta}\right)G_t\right)}\left(\mathrm{I}+t\nabla{\theta}\right)^{-1}\right]:\nabla{\left(\left(\mathrm{I}+t\nabla{\theta}\right)\varphi\right)}\left(\mathrm{I}+t\nabla{\theta}\right)^{-1}\\=\int_{\Omega_0}\left(\mathrm{I}+t\nabla{\theta}^{\top}\right)f_{t}\mathrm{J}_{t}\cdot\varphi, \qquad \forall \varphi\in\HH^{1}_{\mathrm{D}}(\Omega_0,\R^d),
\end{multline*}
and $\mathrm{prox}_{\iota_{\mathcal{K}^{1}(\Omega_{0})}}$ is the proximal operator associated with the Signorini indicator function $\iota_{\mathcal{K}^{1}(\Omega_0)}$ considered on the perturbed Hilbert space $(\HH^{1}_{\mathrm{D}}(\Omega_0,\R^d),\dual{\cdot}{\cdot}_{t})$, where $\dual{\cdot}{\cdot}_t$ is the scalar product defined by
\begin{multline*}
    (v_1,v_2)\in\left(\HH^1_{\mathrm{D}}(\Omega_0,\R^d)\right)^2\mapsto\\ \int_{\Omega_0}\mathrm{J}_{t}\mathrm{A}\left[\nabla{\left(\left(\mathrm{I}+t\nabla{\theta}\right)v_1\right)}\left(\mathrm{I}+t\nabla{\theta}\right)^{-1}\right]:~\nabla{\left(\left(\mathrm{I}+t\nabla{\theta}\right)v_2\right)}\left(\mathrm{I}+t\nabla{\theta}\right)^{-1}\in\R.
\end{multline*}

The previous difficulty is solved since the Signorini indicator function $\iota_{\mathcal{K}^{1}(\Omega_0)}$ does not depend on the parameter $t\geq0$. Nevertheless, we face here to a perturbed Hilbert space due to the scalar product~$\dual{\cdot}{\cdot}_{t}$ that is~$t$-dependent, thus we could not apply Theorem~\ref{TheoABC2018}. To overcome this difficulty, let us rewrite Inequality~\eqref{firstsi} as (using the equality $\mathrm{B}:\mathrm{C}\mathrm{D}=\mathrm{B}\mathrm{D}^{\top}:\mathrm{C}$ for all~$\mathrm{B},\mathrm{C},\mathrm{D}\in\R^{d\times d}$)
\begin{multline*}
\displaystyle\int_{\Omega_0}\mathrm{J}_{t}\mathrm{A}\left[\nabla{\left(\left(\mathrm{I}+t\nabla{\theta}\right)\overline{\overline{u}}_t\right)}\left(\mathrm{I}+t\nabla{\theta}\right)^{-1}\right]\left(\mathrm{I}+t\nabla{\theta}^{\top}\right)^{-1}:\nabla{\left(\left(\mathrm{I}+t\nabla{\theta}\right)\left(\varphi-\overline{\overline{u}}_t\right)\right)}\\\geq\int_{\Omega_0}\left(\mathrm{I}+t\nabla{\theta}^{\top}\right)f_{t}\mathrm{J}_{t}\cdot\left(\varphi-\overline{\overline{u}}_t\right), \qquad \forall \varphi\in\mathcal{K}^{1}(\Omega_0),
\end{multline*}
then adding to both members $\dual{\overline{\overline{u}}_t}{\varphi-\overline{\overline{u}}_t}_{\HH^{1}_{\mathrm{D}}(\Omega_0,\R^{d})}$, one deduces that
\begin{multline*}
\displaystyle\dual{\overline{\overline{u}}_t}{\varphi-\overline{\overline{u}}_t}_{\HH^{1}_{\mathrm{D}}(\Omega_0,\R^{d})}\geq\int_{\Omega_0}\left(\mathrm{I}+t\nabla{\theta}^{\top}\right)f_{t}\mathrm{J}_{t}\cdot(\varphi-\overline{\overline{u}}_t)\\-\int_{\Omega_0}\left(\mathrm{J}_{t}\mathrm{A}\left[\nabla{\overline{\overline{u}}_t}\left(\mathrm{I}+t\nabla{\theta}\right)^{-1}\right]\left(\mathrm{I}+t\nabla{\theta}^{\top}\right)^{-1}-\mathrm{A}\nabla{\overline{\overline{u}}_t}\right):\nabla{\left(\varphi-\overline{\overline{u}}_t\right)}\\-t\int_{\Omega_0}\mathrm{J}_{t}\mathrm{A}\left[\nabla{\overline{\overline{u}}_t}\left(\mathrm{I}+t\nabla{\theta}\right)^{-1}\right]\left(\mathrm{I}+t\nabla{\theta}^{\top}\right)^{-1}:\nabla{\left(\nabla{\theta}\left(\varphi-\overline{\overline{u}}_t\right)\right)}\\-t\int_{\Omega_0}\mathrm{J}_{t}\mathrm{A}\left[\nabla{\left(\nabla{\theta}\overline{\overline{u}}_t\right)}\left(\mathrm{I}+t\nabla{\theta}\right)^{-1}\right]\left(\mathrm{I}+t\nabla{\theta}^{\top}\right)^{-1}:\nabla{\left(\varphi-\overline{\overline{u}}_t\right)}\\-t^2\int_{\Omega_0}\mathrm{J}_{t}\mathrm{A}\left[\nabla{\left(\nabla{\theta}\overline{\overline{u}}_t\right)}\left(\mathrm{I}+t\nabla{\theta}\right)^{-1}\right]\left(\mathrm{I}+t\nabla{\theta}^{\top}\right)^{-1}:\nabla{\left(\nabla{\theta}\left(\varphi-\overline{\overline{u}}_t\right)\right)}, \qquad \forall \varphi\in\mathcal{K}^{1}(\Omega_0).
\end{multline*}
Thus $\overline{\overline{u}}_t$ is also expressed as
\begin{equation*}\label{fonctionavant}
     \displaystyle \overline{\overline{u}}_t=\mathrm{prox}_{\iota_{\mathcal{K}^{1}(\Omega_0)}}(E_{t}),
\end{equation*}
where $E_{t}\in\HH^{1}_{\mathrm{D}}(\Omega_0,\R^d)$ stands for the unique solution to the parameterized variational equality
\begin{multline}\label{perturprob}
\displaystyle\dual{E_t}{\varphi}_{\HH^{1}_{\mathrm{D}}(\Omega_0,\R^{d})}=\int_{\Omega_0}\left(\mathrm{I}+t\nabla{\theta}^{\top}\right)f_{t}\mathrm{J}_{t}\cdot\varphi\\-\int_{\Omega_0}\left(\mathrm{J}_{t}\mathrm{A}\left[\nabla{\overline{\overline{u}}_t}\left(\mathrm{I}+t\nabla{\theta}\right)^{-1}\right]\left(\mathrm{I}+t\nabla{\theta}^{\top}\right)^{-1}-\mathrm{A}\nabla{\overline{\overline{u}}_t}\right):\nabla{\varphi}\\-t\int_{\Omega_0}\mathrm{J}_{t}\mathrm{A}\left[\nabla{\overline{\overline{u}}_t}\left(\mathrm{I}+t\nabla{\theta}\right)^{-1}\right]\left(\mathrm{I}+t\nabla{\theta}^{\top}\right)^{-1}:\nabla{\left(\nabla{\theta}\varphi\right)}\\-t\int_{\Omega_0}\mathrm{J}_{t}\mathrm{A}\left[\nabla{\left(\nabla{\theta}\overline{\overline{u}}_t\right)}\left(\mathrm{I}+t\nabla{\theta}\right)^{-1}\right]\left(\mathrm{I}+t\nabla{\theta}^{\top}\right)^{-1}:\nabla{\varphi}\\-t^2\int_{\Omega_0}\mathrm{J}_{t}\mathrm{A}\left[\nabla{\left(\nabla{\theta}\overline{\overline{u}}_t\right)}\left(\mathrm{I}+t\nabla{\theta}\right)^{-1}\right]\left(\mathrm{I}+t\nabla{\theta}^{\top}\right)^{-1}:\nabla{\left(\nabla{\theta}\varphi\right)}, \qquad \forall \varphi\in\HH^{1}_{\mathrm{D}}(\Omega_0,\R^d),
\end{multline}
and where $\mathrm{prox}_{\iota_{\mathcal{K}^{1}(\Omega_0)}}$ is the proximal operator associated with the Signorini indicator function~$\iota_{\mathcal{K}^{1}(\Omega_0)}$
considered on the Hilbert space $(\HH^{1}_{\mathrm{D}}(\Omega_0,\R^d),\dual{\cdot}{\cdot}_{\HH^{1}_{\mathrm{D}}(\Omega_0,\R^d)})$, which is~$t$-independent.

\begin{myRem}\normalfont
    Note that, for the parameter $t=0$, one has $E_0=F_0\in\HH^{1}_{\mathrm{D}}(\Omega_0,\R^d)$, which is the unique solution to the Dirichlet-Neumann problem 
    \begin{equation*}
    \arraycolsep=2pt
     \left\{
    \begin{array}{rcll}
    -\mathrm{div}(\mathrm{A}\mathrm{e}(F_0)) & = & f   & \text{ in } \Omega_0 , \\
    F_0 & = & 0  & \text{ on } \Gamma_{\mathrm{D}} ,\\
    \mathrm{A}\mathrm{e}(F_0)\nn & = & 0  & \text{ on } \Gamma_{\mathrm{S}_0}.
    \end{array}
    \right.
\end{equation*}
\end{myRem}

Now the next step is to derive the differentiability of the map $t\in\R_{+} \mapsto E_{t} \in \HH^{1}_{\mathrm{D}}(\Omega_0,\R^d)$ at~$t=0$. For this purpose let us recall first that (see, e.g.,~\cite{HENROT}):
\begin{enumerate}[label={\rm (\roman*)}]
    \item the map $t\in\R_{+} \mapsto \mathrm{J}_{t}\in\LL^{\infty}(\R^{d})$ is differentiable at $t=0$ with derivative given by~$\mathrm{div}(\theta)$;\label{diff11}
    \item the map $t\in\R_{+} \mapsto \left(\mathrm{I}+t\nabla{\theta}\right)^{-1}\in\LL^{\infty}(\R^{d},\R^{d \times d})$ is differentiable at $t=0$ with derivative given by~$-\nabla{\theta}$;\label{diff12}
    \item the map $t\in\R_{+} \mapsto \left(\mathrm{I}+t\nabla{\theta}^{\top}\right)^{-1}\in\LL^{\infty}(\R^{d},\R^{d \times d})$ is differentiable at $t=0$ with derivative given by~$-\nabla{\theta}^{\top}$;\label{diff13}
    \item the map $t\in\R_{+} \mapsto \left(\mathrm{I}+t\nabla{\theta}^{\top}\right)f_{t}\mathrm{J}_{t}\in\LL^{2}(\R^{d},\R^d)$ is differentiable at $t=0$ with derivative given by~$f\mathrm{div}(\theta)+\nabla{f}\theta+\nabla{\theta}^{\top}f$.\label{diff14}
\end{enumerate}

\begin{myLem}\label{lem876}
The map $t\in\R_{+} \mapsto E_{t} \in \HH^{1}_{\mathrm{D}}(\Omega_0,\R^d)$ is differentiable at $t=0$ and its derivative, denoted by $E'_{0} \in \HH^{1}_{\mathrm{D}}(\Omega_0,\R^d)$, is the unique solution to the variational equality given by
\begin{multline}\label{Woderov}
    \dual{E'_{0}}{\varphi}_{\HH^{1}_{\mathrm{D}}(\Omega_0,\R^{d})}=\int_{\Omega_{0}}\left(f\mathrm{div}(\theta)+\nabla{f}\theta+\nabla{\theta}^{\top}f\right)\cdot \varphi\\+\int_{\Omega_{0}}\left(\left(\mathrm{A}\mathrm{e}(u_0)\right)\nabla{\theta}^{\top}+\mathrm{A}\left(\nabla{u_0}\nabla{\theta}\right)-\mathrm{div}(\theta)\mathrm{A}\mathrm{e}(u_0)\right):\nabla{\varphi}\\-\int_{\Omega_0}\mathrm{A}\mathrm{e}(u_0):\mathrm{e}\left(\nabla{\theta}\varphi\right)-\int_{\Omega_0}\mathrm{A}\mathrm{e}\left(\nabla{\theta}u_0\right):\mathrm{e}\left(\varphi\right), \qquad\forall \varphi\in\HH^{1}_{\mathrm{D}}(\Omega_0,\R^{d}).
\end{multline}
\end{myLem}

\begin{proof}
Using the Riesz representation theorem, we denote by $Z \in\HH^{1}_{\mathrm{D}}(\Omega_0,\R^d)$ the unique solution to the above variational equality~\eqref{Woderov}. From linearity and using differentiability results~\ref{diff11},~\ref{diff12}, \ref{diff13}, \ref{diff14}, one gets
\begin{multline*}
    \left\|\frac{E_{t}-E_{0}}{t}-Z\right\|_{\HH^{1}_{\mathrm{D}}(\Omega_0,\R^{d})}\leq \\C(\Omega_0, d, \mathrm{A}, \theta) \biggl(\left\|\frac{\left(\mathrm{I}+t\nabla{\theta}^{\top}\right)f_{t}\mathrm{J}_{t}-f}{t}-\left(f\mathrm{div}(\theta)+\nabla{f}\theta+\nabla{\theta}^{\top}f\right)\right\|_{\LL^2(\R^d,\R^d)}\\+\left\|\overline{\overline{u}}_t-u_0\right\|_{\HH^{1}_{\mathrm{D}}(\Omega_0,\R^{d})}+\frac{o(t)}{t} \left\|\overline{\overline{u}}_t\right\|_{\HH^{1}_{\mathrm{D}}(\Omega_0,\R^{d})}\biggl),
\end{multline*}
for all $t\geq0$ sufficiently small, where $C(\Omega_0,\mathrm{A},\theta,d)>0$ is a constant which depends on~$\Omega_0,\mathrm{A},\theta,d$, and where $o$ stands for the standard Bachmann–Landau notation, with $\frac{|o(t)|}{t}\rightarrow 0$ when~$t\rightarrow0^{+}$.  Therefore, to conclude the proof, we only need to prove the continuity of the map $t\in\R_{+} \mapsto \overline{\overline{u}}_t\in~\HH^{1}_{\mathrm{D}}(\Omega_0,\R^{d})$ at $t=0$. For this purpose let us take $\varphi=u_{0}$ in the variational formulation of $\overline{\overline{u}}_t$ and~$\varphi=\overline{\overline{u}}_t$ in the variational formulation of $u_{0}$ to get that
\begin{multline*}
    \left\|\overline{\overline{u}}_t-u_0\right\|_{\HH^{1}_{\mathrm{D}}(\Omega_0,\R^{d})}\leq\\ C(\Omega_0,\mathrm{A},\theta,d)\left(\left\|\left(\mathrm{I}+t\nabla{\theta}^{\top}\right)f_{t}\mathrm{J}_{t}-f\right\|_{\LL^2(\R^d,\R^d)}+\left\|\overline{\overline{u}}_t\right\|_{\HH^{1}_{\mathrm{D}}(\Omega_0,\R^{d})}\left(t+o(t)\right)\right),
\end{multline*}
for all $t\geq0$ sufficiently small. Then, to conclude the proof, we only need to prove that the map~$t\in\R_{+} \mapsto\left\|\overline{\overline{u}}_t\right\|_{\HH^{1}_{\mathrm{D}}(\Omega_0,\R^{d})} \in\R$ is bounded for $t\geq0$ sufficiently small. Let us take $\varphi=0$ in the variational formulation of $\overline{\overline{u}}_t$ to get that
\begin{multline*}
\left\|\overline{\overline{u}}_t\right\|^2_{\HH^{1}_{\mathrm{D}}(\Omega_0,\R^{d})}\leq C(\Omega_0,\mathrm{A},\theta,d)\left\|\left(\mathrm{I}+t\nabla{\theta}^{\top}\right)f_{t}\mathrm{J}_{t}\right\|_{\LL^2(\R^d,\R^d)}\left\|\overline{\overline{u}}_t\right\|_{\HH^{1}_{\mathrm{D}}(\Omega_0,\R^{d})}\\+C(\Omega_0,\mathrm{A},\theta,d)\left\|\overline{\overline{u}}_t\right\|^2_{\HH^{1}_{\mathrm{D}}(\Omega_0,\R^{d})}\left(t+o(t)\right),
\end{multline*}
for all $t\geq0$ sufficiently small. Thus, one deduces
$$
\left\|\overline{\overline{u}}_t\right\|_{\HH^{1}_{\mathrm{D}}(\Omega_0,\R^{d})}\leq \frac{C(\Omega_0,\mathrm{A},\theta,d)(\left\|\left(\mathrm{I}+t\nabla{\theta}^{\top}\right)f_{t}\mathrm{J}_{t}\right\|_{\LL^2(\R^d,\R^d)}}{1-C(\Omega_0,\mathrm{A},\theta,d)\left(t+o(t)\right)},
$$
for all $t\geq0$ sufficiently small, and using the continuity of the map $t\in\R_{+} \mapsto (\mathrm{I}+t\nabla{\theta}^{\top})f_{t}\mathrm{J}_{t}\in\LL^{2}(\R^{d},\R^d)$, one concludes the proof.
\end{proof}

\subsection{Twice epi-differentiability, material and shape derivatives}\label{sectionepidiff}

Consider the framework of Subsection~\ref{setting}. Our goal in this subsection is to characterize the derivative of the map~$t\in\R_{+} \mapsto \overline{\overline{u}}_{t} \in \HH^{1}_{\mathrm{D}}(\Omega_0,\R^{d})$ at~$t=0$, denoted $\overline{\overline{u}}'_{0}\in\HH^{1}_{\mathrm{D}}(\Omega_0,\R^{d})$, then to characterize the material derivative $\overline{u}'_0$ (that is the derivative of the map~$t\in\R_{+} \mapsto \overline{u}_{t} \in \HH^{1}_{\mathrm{D}}(\Omega_0,\R^{d})$ at~$t=0$), and finally to conclude with an expression of the shape directional derivative defined by~$u'_{0}:=\overline{u}'_0-\nabla{u_0}\theta$ (which roughly corresponds to the derivative of the map~$t\in\R_{+} \mapsto u_{t} \in \HH^{1}(\Omega_t,\R^d)$ at~$t=0$).

To this aim, our idea is to use Proposition~\ref{TheoABC2018}. Since we have expressed $\overline{\overline{u}}_t=\mathrm{prox}_{\iota_{\mathcal{K}^{1}(\Omega_0)}}(E_{t})$ and characterized in Lemma~\ref{lem876} the derivative of the map~$t\in\R_{+} \mapsto E_{t} \in \HH^{1}_{\mathrm{D}}(\Omega_0,\R^{d})$ at~$t=0$, the next step is to investigate the twice epi-differentiability of the Signorini indicator function~$\iota_{\mathcal{K}^{1}(\Omega_0)}$ at~$u_0\in\mathcal{K}^{1}(\Omega_0)$ for $E_0-u_0\in\mathrm{N}_{\mathcal{K}^{1}(\Omega_0)}(u_0)$, where $\mathrm{N}_{\mathcal{K}^{1}(\Omega_0)}(u_0)$ is the normal cone to~$\mathcal{K}^{1}(\Omega_0)$ at~$u_0$ (see Definition~\ref{conenormall}) which coincides with the subdifferential of $\iota_{\mathcal{K}^{1}(\Omega_0)}$ at $u_0$ (see Example~\ref{ExamNormal}). For this purpose, we will use Lemma~\ref{epipoly} that characterizes the twice epi-differentiability of an indicator function, but it requires that the set~$\mathcal{K}^{1}(\Omega_0)$ is polyhedric (see Definition~\ref{polyset}). In~\cite[Lemma 5.2.9 p.116]{MAURY} the author proves that~$\mathcal{K}^{1}(\Omega_0)$ is polyhedric using some concepts from convex analysis such as the tangential cone (see Definition~\ref{tangentcone}), and also from capacity theory (see Appendix~\ref{capacity} for some reminders and, in particular, Definition~\ref{quasieve} for the notion of \textit{quasi everywhere}).

\begin{myLem}
The nonempty closed convex subset $\mathcal{K}^{1}(\Omega_0)$ of $\HH^{1}_{\mathrm{D}}(\Omega_0,\R^{d})$ is polyhedric at $u_0\in\mathcal{K}^{1}(\Omega_0)$ for $E_0-u_0\in\mathrm{N}_{\mathcal{K}^{1}(\Omega_0)}(u_0)$, and one has
\begin{multline*}
\mathrm{T}_{\mathcal{K}^{1}(\Omega_0)}(u_0)\cap\left(\R\left(E_0-u_0\right)\right)^{\perp}\\=\left\{ \varphi\in\HH^{1}_{\mathrm{D}}(\Omega_0,\R^d) \mid \varphi_\nn\leq0 \text{ \textit{q.e.} on } \Gamma^{{u_0}_{\nn}}_{\mathrm{S}_0} \text{ and } \dual{E_0-u_0}{\varphi}_{\HH^{1}_{\mathrm{D}}(\Omega_0,\R^d)}=0\right\},
\end{multline*}
where $\mathrm{T}_{\mathcal{K}^{1}(\Omega_0)}(u_0)$ is the tangent cone to $\mathcal{K}^{1}(\Omega_0)$ at $u_0$, $\Gamma^{{u_0}_{\nn}}_{\mathrm{S}_0}:=\left\{s\in\Gamma_{\mathrm{S}_0} \mid {u_0}_{\nn}(s)=0 \right\}$ and \textit{q.e.} means quasi everywhere.
\end{myLem}

Since $\mathcal{K}^{1}(\Omega_0)$ is polyhedric at $u_0\in\mathcal{K}^{1}(\Omega_0)$ for $E_0-u_0\in\mathrm{N}_{\mathcal{K}^{1}(\Omega_0)}(u_0)$, one can now deduce from Lemma~\ref{epipoly} that $\iota_{\mathcal{K}^{1}(\Omega_0)}$ is twice epi-differentiable.

\begin{myLem}\label{epidiffindicsi}
 The Signorini indicator function $\iota_\mathrm{\mathcal{K}^{1}(\Omega_0)}$ is twice epi-differentiable at $u_0\in\mathcal{K}^{1}(\Omega_0)$ for $E_0-u_0\in\mathrm{N}_{\mathcal{K}^{1}(\Omega_0)}(u_0)$ and
$$
\mathrm{d}_{e}^{2}\iota_{\mathcal{K}^{1}(\Omega_0)}(u_0|E_0-u_0)=\mathrm{\iota}_{\mathrm{T}_{\mathcal{K}^{1}(\Omega_0)}(u_0)\cap\left(\R\left(E_0-u_0\right)\right)^{\perp}}.
$$
\end{myLem}

The twice epi-differentiability of the Signorini indicator function allows us to apply Proposition~\ref{TheoABC2018} in order to state and prove the first main result of this paper that characterizes the material derivative.

\begin{myTheorem}[Material derivative]\label{materialderiv1}
The map
$t\in\R_{+} \mapsto \overline{u}_{t} \in \HH^{1}_{\mathrm{D}}(\Omega_0,\R^d)$ is differentiable at~$t=0$, and its derivative~$\overline{u}'_{0} \in\mathrm{T}_{\mathcal{K}^{1}(\Omega_0)}(u_0)\cap\left(\R\left(E_0-u_0\right)\right)^{\perp}+\nabla{\theta}u_0$ is the unique solution to the variational inequality
\begin{multline}\label{inequalityofmaterialderiv1}
\dual{\overline{u}'_{0}}{\varphi-\overline{u}'_{0}}_{\HH^{1}_{\mathrm{D}}(\Omega_0,\R^{d})}\geq-\int_{\Omega_0}\mathrm{div}\left(\mathrm{div}\left(\mathrm{A}\mathrm{e}(u_0)\right)\theta^{\top}\right)\cdot\left(\varphi-\overline{u}'_0\right)\\+\int_{\Omega_0}\left(\left(\mathrm{A}\mathrm{e}(u_0)\right)\nabla{\theta}^{\top}+\mathrm{A}\left(\nabla{u_0}\nabla{\theta}\right)-\mathrm{div}(\theta)\mathrm{A}\mathrm{e}(u_0)\right):\nabla{\left(\varphi-\overline{u}'_{0}\right)}\\-\dual{\mathrm{A}\mathrm{e}(u_0)\nn}{\nabla{\theta}\left(\varphi-\overline{u}'_{0}\right)}_{\HH^{-1/2}(\Gamma_0,\R^d)\times\HH^{1/2}(\Gamma_0,\R^d)},
\end{multline}
for all $\varphi\in\mathrm{T}_{\mathcal{K}^{1}(\Omega_0)}(u_0)\cap\left(\R\left(E_0-u_0\right)\right)^{\perp}+\nabla{\theta}u_0$,
where
\begin{multline*}
    \mathrm{T}_{\mathcal{K}^{1}(\Omega_0)}(u_0)\cap\left(\R\left(E_0-u_0\right)\right)^{\perp}+\nabla{\theta}u_0=\\ \left\{ \varphi\in\HH^{1}_{\mathrm{D}}(\Omega_0,\R^d) \mid \varphi_\nn\leq\left(\nabla{\theta}u_0\right)_\nn \text{ \textit{q.e.} on } \Gamma^{{u_0}_{\nn}}_{\mathrm{S}_0} \text{ and } \dual{E_0-u_0}{\varphi-\nabla{\theta}u_0}_{\HH^{1}_{\mathrm{D}}(\Omega_0,\R^d)}=0\right\}.    
\end{multline*}
\end{myTheorem}

\begin{proof}
For all $t\geq0$, $\overline{\overline{u}}_t\in\HH^{1}_{\mathrm{D}}(\Omega_0,\R^{d})$ is given by    
$$
\overline{\overline{u}}_t=\mathrm{prox}_{\iota_{\mathcal{K}^{1}(\Omega_0)}}(E_{t}),
$$
where $E_{t}\in\HH^{1}(\Omega_{0},\R^d)$ stands for the unique solution to the parameterized variational equality~\eqref{perturprob}.
Moreover the map $t\in\R_{+} \mapsto E_{t} \in \HH^{1}_{\mathrm{D}}(\Omega_0,\R^d)$ is differentiable at $t=0$ with its derivative~$E'_0~\in\HH^{1}_{\mathrm{D}}(\Omega_0,\R^{d})$ solution to the variational equality~\eqref{Woderov}. Therefore, from Lemma~\ref{epidiffindicsi} one can apply Proposition~\ref{TheoABC2018} to deduce the differentiability of the map $t\in\R_{+} \mapsto \overline{\overline{u}}_t \in \HH^{1}_{\mathrm{D}}(\Omega_0,\R^{d})$, with its derivative~$\overline{\overline{u}}'_0\in\HH^{1}_{\mathrm{D}}(\Omega_0,\R^{d})$ given by
$$
\overline{\overline{u}}'_0=\mathrm{prox}_{\mathrm{d}_{e}^{2}\iota_{\mathcal{K}^{1}(\Omega_0)}(u_0|E_0-u_0)}(E'_{0}).
$$
Since $\overline{u}_{t}=\left(\mathrm{I}+t\nabla{\theta}\right)\overline{\overline{u}}_t$, then one deduces that $$
\overline{u}'_{0}=\mathrm{prox}_{\mathrm{d}_{e}^{2}\iota_{\mathcal{K}^{1}(\Omega_0)}(u_0|E_0-u_0)}(E'_{0})+\nabla{\theta}u_0,
$$
which, from definition of the proximal operator (see~Definition~\ref{proxi}), leads to
\begin{multline*}  \dual{E'_0-\overline{u}'_0+\nabla{\theta}u_0}{v-\overline{u}'_0 +\nabla{\theta}u_0 }_{\HH^{1}_{\mathrm{D}}(\Omega_0,\R^{d})}\leq \\\mathrm{d}_{e}^{2}\iota_{\mathcal{K}^{1}(\Omega_0)}(u_0|E_0-u_0)(v)-\mathrm{d}_{e}^{2}\iota_{\mathcal{K}^{1}(\Omega_0)}(u_0|E_0-u_0)(\overline{u}'_0-\nabla{\theta}u_0),
\end{multline*}
for all $v\in\HH^{1}_{\mathrm{D}}(\Omega_0,\R^{d})$, i.e.,
$$
\dual{E'_0-\overline{u}'_0+\nabla{\theta}u_0}{\varphi-\overline{u}'_0 }_{\HH^{1}_{\mathrm{D}}(\Omega_0,\R^{d})}\leq 0,
$$
for all $\varphi\in\mathrm{T}_{\mathcal{K}^{1}(\Omega_0)}(u_0)\cap\left(\R\left(E_0-u_0\right)\right)^{\perp}+\nabla{\theta}u_0$. Using the variational equality satisfied by $E'_0$ (see~\eqref{Woderov}), one gets
\begin{multline*}
  \dual{\overline{u}'_0-\nabla{\theta}u_0}{\varphi-\overline{u}'_0 }_{\HH^{1}_{\mathrm{D}}(\Omega_0,\R^{d})}\geq \int_{\Omega_{0}}\left(f\mathrm{div}(\theta)+\nabla{f}\theta+\nabla{\theta}^{\top}f\right)\cdot\left(\varphi-\overline{u}'_0\right)\\+\int_{\Omega_{0}}\left(\left(\mathrm{A}\mathrm{e}(u_0)\right)\nabla{\theta}^{\top}+\mathrm{A}\left(\nabla{u_0}\nabla{\theta}\right)-\mathrm{div}(\theta)\mathrm{A}\mathrm{e}(u_0)\right):\nabla{\left(\varphi-\overline{u}'_0\right)}\\-\int_{\Omega_0}\mathrm{A}\mathrm{e}(u_0):\mathrm{e}\left(\nabla{\theta}\left(\varphi-\overline{u}'_0\right)\right)-\int_{\Omega_0}\mathrm{A}\mathrm{e}\left(\nabla{\theta}u_0\right):\mathrm{e}\left(\varphi-\overline{u}'_0\right),  
\end{multline*}
for all $\varphi\in\mathrm{T}_{\mathcal{K}^{1}(\Omega_0)}(u_0)\cap\left(\R\left(E_0-u_0\right)\right)^{\perp}+\nabla{\theta}u_0$, which is also (see the notations introduced at the beginning of Section~\ref{BVP})
\begin{multline*}
\dual{\overline{u}'_0-\nabla{\theta}u_0}{\varphi-\overline{u}'_0 }_{\HH^{1}_{\mathrm{D}}(\Omega_0,\R^{d})}\geq \int_{\Omega_{0}}\mathrm{div}(f\theta^{\top})\cdot\left(\varphi-\overline{u}'_0\right)+\int_{\Omega_0}f\cdot\nabla{\theta}\left(\varphi-\overline{u}'_0\right)\\+\int_{\Omega_{0}}\left(\left(\mathrm{A}\mathrm{e}(u_0)\right)\nabla{\theta}^{\top}+\mathrm{A}\left(\nabla{u_0}\nabla{\theta}\right)-\mathrm{div}(\theta)\mathrm{A}\mathrm{e}(u_0)\right):\nabla{\left(\varphi-\overline{u}'_0\right)}\\-\int_{\Omega_0}\mathrm{A}\mathrm{e}(u_0):\mathrm{e}\left(\nabla{\theta}\left(\varphi-\overline{u}'_0\right)\right)-\int_{\Omega_0}\mathrm{A}\mathrm{e}\left(\nabla{\theta}u_0\right):\mathrm{e}\left(\varphi-\overline{u}'_0\right),
\end{multline*}
for all $\varphi\in\mathrm{T}_{\mathcal{K}^{1}(\Omega_0)}(u_0)\cap\left(\R\left(E_0-u_0\right)\right)^{\perp}+\nabla{\theta}u_0$. Using the divergence formula (see Proposition~\ref{div}) and the equality~$\mathrm{div}\left(\mathrm{A}\mathrm{e}(u_0)\right)=-f$ in $\HH^1(\Omega_0,\R^d)$, we obtain the result.
\end{proof}

In~\cite{CHAUDET2},~\cite[Chapter 5 Section 5.2 p.111]{MAURY} and~\cite[Chapter 4 Section 4.6 p.205]{SOKOZOL}, the authors get a similar result using the conical differentiability of the projection operator. Since $\mathcal{K}^{1}(\Omega_0)$ is polyhedric at $u_0\in\mathcal{K}^{1}(\Omega_0)$ for $E_0-u_0\in\mathrm{N}_{\mathcal{K}^{1}(\Omega_0)}(u_0)$, then from Mignot's theorem (see~\cite{MIGNOT}) the projection operator on~$\mathcal{K}^{1}(\Omega_0)$ is conically differentiable at~$u_0$ for~$E_0-u_0$, and its conical derivative is given by~$\mathrm{proj}_{\mathrm{T}_{\mathcal{K}^{1}(\Omega_0)}(u_0)\cap\left(\R\left(E_0-u_0\right)\right)^{\perp}}(E'_{0})$, which is exactly $\mathrm{prox}_{\mathrm{\iota}_{\mathrm{T}_{\mathrm{\mathcal{K}^{1}(\Omega_0)}}(u_0)\cap\left(\R \left(E_0-u_0\right)\right)^{\perp}}}(E'_{0})$. Nevertheless, to the best of our knowledge, it has not been observed in the literature that it was possible to improve this result under additional assumptions, in order to characterize the material derivative as weak solution to a boundary value problem. Indeed, as mentioned in~\cite{CHAUDET2,HINTERMULLERLAURAIN}, it is possible to replace \textit{q.e.} in the set $\mathrm{T}_{\mathcal{K}^{1}(\Omega_0)}(u_0)\cap\left(\R\left(E_0-u_0\right)\right)^{\perp}$ by \textit{\textit{a.e.}} under some hypotheses, like, for instance, if $\mathrm{\Gamma^{{u_0}_{\nn}}_{\mathrm{S}_0}}=\overline{\mathrm{int}_{\Gamma_{\mathrm{S}_0}}(\mathrm{\Gamma^{{u_0}_{\nn}}_{\mathrm{S}_0}})}$. Moreover, if we assume that the decomposition~$\Gamma_{\mathrm{D}}\cup\Gamma_{\mathrm{S}_0}$ of~$\Gamma_0$ is consistent (see Definition~\ref{regulieresens2} with $\Gamma_{\mathrm{S}_{\mathrm{S}}}:=\Gamma_{\mathrm{S}_0}$ and $w=0$) and some regularity assumptions on $u_0$ and $\theta$, then the material derivative can be characterized as weak solution to a Signorini problem.

\begin{myCor}\label{materialderiv2}
Assume that the decomposition $\Gamma_{\mathrm{D}}\cup\Gamma_{\mathrm{S}_0}$ of $\Gamma_0$ is consistent,~$u_{0}\in\HH^{3}(\Omega_{0},\R^d)$ and~$\mathrm{\Gamma^{{u_0}_{\nn}}_{\mathrm{S}_0}}=\overline{\mathrm{int}_{\Gamma_{\mathrm{S}_0}}(\mathrm{\Gamma^{{u_0}_{\nn}}_{\mathrm{S}_0}})}$. Then the material derivative
$\overline{u}'_0\in\mathrm{T}_{\mathcal{K}^{1}(\Omega_0)}(u_0)\cap\left(\R\left(E_0-u_0\right)\right)^{\perp}+\nabla{\theta}u_0$ is the unique weak solution to the Signorini problem
\begin{equation*}
\footnotesize
{\arraycolsep=2pt
\left\{
\begin{array}{rcll}
-\mathrm{div}(\mathrm{A}\mathrm{e}(\overline{u}'_0)) & = & -\mathrm{div}(\mathrm{A}\mathrm{e}(\nabla{u_0}\theta))   & \text{ in } \Omega_0 , \\
\overline{u}'_0 & = & 0  & \text{ on } \Gamma_{\mathrm{D}} ,\\
\sigma_{\tau}(\overline{u}'_0) & = & h^m(\theta)_{\tau}  & \text{ on } \Gamma_{\mathrm{S}_0} ,\\
 \sigma_{\nn}(\overline{u}'_0) & = & h^m(\theta)_{\nn}  & \text{ on } \Gamma_{\mathrm{S}^{{u_0}_{\nn}}_{0,\mathrm{N}}} ,\\
{\overline{u}'_{0}}_{\nn}& = & \left(\nabla{\theta}u_0\right)_{\nn}  & \text{ on } \Gamma_{\mathrm{S}^{{u_0}_{\nn}}_{0,\mathrm{D}}} ,\\
{\overline{u}'_{0}}_{\nn}\leq\left(\nabla{\theta}u_0\right)_{\nn}, \sigma_{\nn}(\overline{u}'_0)\leq h^m(\theta)_{\nn} \text{ and } \left( {\overline{u}'_{0}}_{\nn}-\left(\nabla{\theta}u_0\right)_{\nn}\right)\left(\sigma_{\nn}(\overline{u}'_0)-h^m(\theta)_{\nn}\right) & = & 0  & \text{ on } \Gamma_{\mathrm{S}^{{u_0}_{\nn}}_{0,\mathrm{S}}},
\end{array}
\right.}
\end{equation*}
where $h^m(\theta):=((\mathrm{A}\mathrm{e}(u_0))\nabla{\theta}^{\top}+\mathrm{A}(\nabla{u_0}\nabla{\theta})-\sigma_{\nn}(u_0)(\mathrm{div}(\theta)\mathrm{I}+\nabla{\theta}^{\top}))\nn\in\LL^2(\Gamma_{\mathrm{S}_0},\R^d)$, and~$\Gamma_{\mathrm{S}_0}$ is decomposed, up to a null set, as $\Gamma_{\mathrm{S}^{{u_0}_{\nn}}_{0,\mathrm{N}}}\cup\Gamma_{\mathrm{S}^{{u_0}_{\nn}}_{0,\mathrm{D}}}\cup\Gamma_{\mathrm{S}^{{u_0}_{\nn}}_{0,\mathrm{S}}}$, where
$$
\begin{array}{l}
\Gamma_{\mathrm{S}^{{u_0}_{\nn}}_{0,\mathrm{N}}}:=\left\{s\in\Gamma_{\mathrm{S}_0} \mid  {u_{0}}_{\nn}(s)\neq 0 \right \}, \\
\Gamma_{\mathrm{S}^{{u_0}_{\nn}}_{0,\mathrm{D}}}:=\left\{s\in\Gamma_{\mathrm{S}_0} \mid  {u_{0}}_{\nn}(s) = 0 \text{ and } \sigma_{\nn}(u_0)(s) < 0\right \}, \\
\Gamma_{\mathrm{S}^{{u_0}_{\nn}}_{0,\mathrm{S}}}:=\left\{s\in\Gamma_{\mathrm{S}_0} \mid  {u_{0}}_{\nn}(s) = 0 \text{ and } \sigma_{\nn}(u_0)(s)= 0\right \}.
\end{array}
$$

\end{myCor}

\begin{proof}
Since $u_{0}\in\HH^{2}(\Omega_{0},\R^d)$ and $\theta \in \mathcal{C}_{\mathrm{D}}^{2,\infty}(\R^{d},\R^{d})$, one deduces that 
    $$
\mathrm{div}\left(\left(\mathrm{A}\mathrm{e}(u_0)\right)\nabla{\theta}^{\top}+\mathrm{A}\left(\nabla{u_0}\nabla{\theta}\right)-\mathrm{div}(\theta)\mathrm{A}\mathrm{e}(u_0)\right)\in\LL^2(\Omega_0,\R^d).
$$
Moreover, since $\mathrm{A}\mathrm{e}(u_0)\nn\in\LL^2(\Gamma_0,\R^d)$ and that the decomposition~$\Gamma_{\mathrm{D}}\cup\Gamma_{\mathrm{S}_0}$ of $\Gamma_0$ is consistent, then $u_0$ is a (strong) solution to the Signorini problem~\eqref{PbperturbSignorini} for the parameter $t=0$ (see Proposition~\ref{EquiSignorini}).
Thus $\sigma_\tau(u_0)=0$ \textit{a.e.} on $\Gamma_{\mathrm{S}_0}$, and using the divergence formula (see Proposition~\ref{div}) in Inequality~\eqref{inequalityofmaterialderiv1}, we get that
\begin{multline}\label{formulderivmatH2}
\dual{\overline{u}'_0}{\varphi-\overline{u}'_0 }_{\HH^{1}_{\mathrm{D}}(\Omega_0,\R^d)}\geq \int_{\Gamma_{\mathrm{S}_0}}h^{m}(\theta)\cdot\left(\varphi-\overline{u}'_0\right)\\ -\int_{\Omega_0}\mathrm{div}\left(\mathrm{div}(\mathrm{A}\mathrm{e}(u_0))\theta^{\top}+\left(\mathrm{A}\mathrm{e}(u_0)\right)\nabla{\theta}^{\top}+\mathrm{A}\left(\nabla{u_0}\nabla{\theta}\right)-\mathrm{div}(\theta)\mathrm{A}\mathrm{e}(u_0)\right)\cdot\left(\varphi-\overline{u}'_0\right),
\end{multline}
for all $\varphi \in \mathrm{T}_{\mathcal{K}^{1}(\Omega_0)}(u_0)\cap\left(\R\left(E_0-u_0\right)\right)^{\perp}+\nabla{\theta}u_0$. Furthermore, one has $\mathrm{div}\left(\mathrm{A}\mathrm{e}\left(\nabla{u_0}\theta\right)\right)\in\LL^2(\Omega_0,\R^d)$ from $u_0\in\HH^3(\Omega_0,\R^d)$. Thus, using the equality
$$
\mathrm{div}\left(\mathrm{A}\mathrm{e}\left(\nabla{u_0}\theta\right)\right)=\mathrm{div}\left(\mathrm{div}(\mathrm{A}\mathrm{e}(u_0))\theta^{\top}+\left(\mathrm{A}\mathrm{e}(u_0)\right)\nabla{\theta}^{\top}+\mathrm{A}\left(\nabla{u_0}\nabla{\theta}\right)-\mathrm{div}(\theta)\mathrm{A}\mathrm{e}(u_0)\right),
$$
in $\LL^2(\Omega_0,\R^d)$, it follows that
\begin{equation*}
\dual{\overline{u}'_0}{\varphi-\overline{u}'_0 }_{\HH^{1}_{\mathrm{D}}(\Omega_0,\R^d)}\geq -\int_{\Omega_0}\mathrm{div}\left(\mathrm{A}\mathrm{e}\left(\nabla{u_0}\theta\right)\right)\cdot\left(\varphi-\overline{u}'_0\right)
+
\int_{\Gamma_{\mathrm{S}_0}}h^{m}(\theta)\cdot\left(\varphi-\overline{u}'_0\right),
\end{equation*}
for all $\varphi \in \mathrm{T}_{\mathcal{K}^{1}(\Omega_0)}(u_0)\cap\left(\R\left(E_0-u_0\right)\right)^{\perp}+\nabla{\theta}u_0$. Moreover, since
$$\mathcal{H}:=\left\{ v_\nn\in\HH^{1/2}(\Gamma_{\mathrm{S_0}},\R) \mid v\in\HH^1_{\mathrm{D}}(\Omega_0,\R^d) \right\},
$$
is a Dirichlet space (see Example~\ref{DirichletSpace}), then, for all $v\in\HH^1_{\mathrm{D}}(\Omega_0,\R^d)$, $v_\nn\in\HH^{1/2}(\Gamma_{\mathrm{S}},\R)$ admits a unique quasi-continuous representative for the \textit{q.e} equivalence relation (see Proposition~\ref{uniquerepres}), thus it follows that (see~\cite[Remark 3.12 p.13]{CHAUDET2} for details)
\begin{multline*}
\mathrm{T}_{\mathcal{K}^{1}(\Omega_0)}(u_0)\cap\left(\R\left(E_0-u_0\right)\right)^{\perp}\\=\left\{ \varphi\in\HH^{1}_{\mathrm{D}}(\Omega_0,\R^d) \mid \varphi_\nn\leq0 \text{ \textit{a.e.} on } \Gamma^{{u_0}_{\nn}}_{\mathrm{S}_0} \text{ and } \dual{E_0-u_0}{\varphi}_{\HH^{1}_{\mathrm{D}}(\Omega_0,\R^d)}=0\right\}.
\end{multline*}
Furthermore, since $u_0$ is a (strong) solution, it follows from the Signorini unilateral conditions that
$$
\dual{E_0-u_0}{\varphi}_{\HH^{1}_{\mathrm{D}}(\Omega_0,\R^d)}=\int_{\Gamma_0}\mathrm{A}\mathrm{e}(E_0-u_0)\nn\cdot \varphi=-\int_{\Gamma_{\mathrm{S}_0}}\sigma_{\nn}(u_0)\varphi_{\nn}=-\int_{\Gamma_{\mathrm{S}^{{u_0}_{\nn}}_{0,\mathrm{D}}}}\sigma_{\nn}(u_0)\varphi_{\nn}=0,
$$
 for all $\varphi\in\mathrm{T}_{\mathcal{K}^{1}(\Omega_0)}(u_0)\cap\left(\R\left(E_0-u_0\right)\right)^{\perp}$, and since $\sigma_{\nn}(u_0)\varphi_{\nn}\leq0$ \textit{a.e.} on $\Gamma_{\mathrm{S}^{{u_0}_{\nn}}_{0,\mathrm{D}}}$ and $\sigma_{\nn}(u_0)<0$ \textit{a.e.} on~$\Gamma_{\mathrm{S}^{{u_0}_{\nn}}_{0,\mathrm{D}}}$, one deduces that $\varphi_{\nn}=0$ \textit{a.e.} on~$\Gamma_{\mathrm{S}^{{u_0}_{\nn}}_{0,\mathrm{D}}}$. From Subsection~\ref{SectionSignorinicasscalairesansu}, one concludes the proof.
\end{proof}

\begin{myRem}\normalfont
If $\Gamma_0$ is smooth, the best known regularity result guarantees that $u_0 \in \mathcal{C}^{1,1/2}_{\rm loc}$ (see~\cite{ANDER,SCHU}). Nevertheless, we consider here that we are in a favorable situation where $u_0 \in \mathrm{H}^3(\Omega_0,\R^d)$, which can obviously appear for some particular cases. Note that, from the proof of Corollary~\ref{materialderiv2}, one can get, under the weaker assumption~$u_0 \in \HH^2(\Omega_0,\R^d)$, that the material derivative $\overline{u}'_0$ is the solution to the variational inequality~\eqref{formulderivmatH2} which is, from Subsection~\ref{SectionSignorinicasscalairesansu}, the weak formulation of a Signorini problem, with the source term given by~$-\mathrm{div}(\mathrm{div}(\mathrm{A}\mathrm{e}(u_0))\theta^{\top}+(\mathrm{A}\mathrm{e}(u_0))\nabla{\theta}^{\top}+\mathrm{A}(\nabla{u_0}\nabla{\theta})-\mathrm{div}(\theta)\mathrm{A}\mathrm{e}(u_0))\in\LL^2(\Omega_0,\R^d)$.
\end{myRem}

Thanks to Corollary~\ref{materialderiv2}, we are now in a position to characterize the shape derivative.

\begin{myCor}[Shape derivative]\label{shapederiv1}
Consider the framework of Corollary~\ref{materialderiv2} with the additional assumption that $\Gamma_0$ is of class $\mathcal{C}^3$. Then the shape derivative, defined by~$u'_{0}:=\overline{u}'_0-\nabla{u_0}\theta\in\mathrm{T}_{\mathcal{K}^{1}(\Omega_0)}(u_0)\cap\left(\R\left(E_0-u_0\right)\right)^{\perp}+\nabla{\theta}u_0-\nabla{u_0}\theta$ is the unique weak solution to the Signorini problem
\begin{equation*}
\small{\arraycolsep=2pt
\left\{
\begin{array}{rcll}
-\mathrm{div}(\mathrm{A}\mathrm{e}(u'_0)) & = & 0   & \text{ in } \Omega_0 , \\
u'_0 & = & 0  & \text{ on } \Gamma_{\mathrm{D}} ,\\
\sigma_{\tau}(u'_0) & = & h^s(\theta)_{\tau}  & \text{ on } \Gamma_{\mathrm{S}_0} ,\\
\sigma_{\nn}(u'_0) & = & h^s(\theta)_{\nn}  & \text{ on } \Gamma_{\mathrm{S}^{{u_0}_{\nn}}_{0,\mathrm{N}}} ,\\
 {u'_{0}}_{\nn}& = & W(\theta)_{\nn} & \text{ on } \Gamma_{\mathrm{S}^{{u_0}_{\nn}}_{0,\mathrm{D}}} ,\\
{u'_{0}}_{\nn}\leq W(\theta)_{\nn}, \sigma_{\nn}(u'_0)\leq h^s(\theta)_{\nn} \text{ and } \left( {u'_{0}}_{\nn}-W(\theta)_{\nn} \right)\left(\sigma_{\nn}(u'_0)-h^s(\theta)_{\nn}\right) & = & 0  & \text{ on } \Gamma_{\mathrm{S}^{{u_0}_{\nn}}_{0,\mathrm{S}}},
\end{array}
\right.}
\end{equation*}
where:
$$\bullet~W(\theta):=\left(\nabla{\theta}u_0\right)-\left(\nabla{u_0}\theta\right)\in\HH^{1/2}(\Gamma_0,\R^d);$$ 
\begin{multline*}
\bullet~h^s(\theta):=\theta \cdot \nn\left(\partial_{\nn}\left(\mathrm{A}\mathrm{e}(u_0)\nn\right)-\partial_{\nn}\left(\mathrm{A}\mathrm{e}(u_0)\right)\nn\right)+\mathrm{A}\mathrm{e}(u_0)\nabla_{\tau}\left(\theta\cdot\nn\right)-\nabla{\left(\mathrm{A}\mathrm{e}(u_0)\nn\right)}\theta\\-\sigma_{\nn}(u_0)\left(\mathrm{div}_{\tau}(\theta)\mathrm{I}+\nabla{\theta}^{\top}\right)\nn\in\LL^2(\Gamma_{\mathrm{S}_0},\R^d);
    \end{multline*}
and where $\partial_{\nn}\left(\mathrm{A}\mathrm{e}(u_0)\nn\right):=\nabla{\left(\mathrm{A}\mathrm{e}(u_0)\nn\right)}\nn$ stands for the normal derivative of $\mathrm{A}\mathrm{e}(u_0)\nn$,~$\partial_{\nn}\left(\mathrm{A}\mathrm{e}(u_0)\right)$ is the matrix whose the $i$-th line is the transpose of the vector~$\partial_{\nn}\left(\mathrm{A}\mathrm{e}(u_0)_{i}\right):=\nabla{\left(\mathrm{A}\mathrm{e}(u_0)_{i}\right)}\nn$, where~$\mathrm{A}\mathrm{e}(u_0)_{i}$ is the transpose of the~$i$-th line of the matrix~$\mathrm{A}\mathrm{e}(u_0)$, for all~$i\in[[1,d]]$.
\end{myCor}

\begin{proof}
Since $u'_{0}:=\overline{u}'_0-\nabla{u_0}\theta\in\HH^1_{\mathrm{D}}(\Omega_0,\R^d)$, one deduces from the weak formulation of $\overline{u}'_0$ and using the divergence formula that,
\begin{multline*}
\dual{u'_0}{\varphi-u'_0}_{\HH^{1}_{\mathrm{D}}(\Omega_0,\R^d)}\geq \\ \int_{\Omega_0}\left(\mathrm{div}\left(\mathrm{A}\mathrm{e}(u_0)\right)\theta^{\top}+\left(\mathrm{A}\mathrm{e}(u_0)\right)\nabla{\theta}^{\top}+\mathrm{A}\left(\nabla{u_0}\nabla{\theta}\right)-\mathrm{A}\mathrm{e}\left(\nabla{u_0}\theta\right)\right):\nabla{\left(\varphi-u'_{0}\right)}\\-\int_{\Omega_0}\mathrm{div}(\theta)\mathrm{A}\mathrm{e}(u_0):\mathrm{e}\left(\varphi-u'_{0}\right)-\int_{\Gamma_{\mathrm{S}_0}}\mathrm{A}\mathrm{e}(u_0)\nn\cdot\nabla{\theta}\left(\varphi-u'_{0}\right)-\int_{\Gamma_{0}}\left(\theta\cdot\nn\right) \mathrm{div}\left(\mathrm{A}\mathrm{e}(u_0)\right)\cdot\left(\varphi-u'_{0}\right),
\end{multline*} 
 for all $\varphi \in \mathrm{T}_{\mathcal{K}^{1}(\Omega_0)}(u_0)\cap\left(\R\left(E_0-u_0\right)\right)^{\perp}+\nabla{\theta}u_0-\nabla{u_0}\theta$. Moreover, one has
$$\int_{\Omega_0}\mathrm{div}\left(\mathrm{A}\mathrm{e}(u_0)\right)\theta^{\top}:\nabla{v}=\int_{\Omega_0}\mathrm{div}\left(\mathrm{A}\mathrm{e}(u_0)\right)\cdot\nabla{v}\theta=-\int_{\Omega_0}\mathrm{A}\mathrm{e}(u_0):\nabla{\left(\nabla{v}\theta\right)}+\int_{\Gamma_0}\mathrm{A}\mathrm{e}(u_0)\nn\cdot\nabla{v}\theta,
$$
and also
$$
-\int_{\Omega_0}\mathrm{div}(\theta)\mathrm{A}\mathrm{e}(u_0):\mathrm{e}(v)=\int_{\Omega_0}\theta\cdot\nabla{\left(\mathrm{A}\mathrm{e}(u_0):\mathrm{e}(v)\right)}-\int_{\Gamma_0}\theta\cdot\nn\left(\mathrm{A}\mathrm{e}(u_0):\mathrm{e}(v)\right),
$$
for all $v\in\mathcal{C}^{\infty}(\overline{\Omega_{0}},\R^d)$.
Therefore, using the equality
$$
\left(\left(\mathrm{A}\mathrm{e}(u_0)\right)\nabla{\theta}^{\top}+\mathrm{A}\left(\nabla{u_0}\nabla{\theta}\right)-\mathrm{A}\mathrm{e}\left(\nabla{u_0}\theta\right)\right):\nabla{v}+\theta\cdot\nabla{\left(\mathrm{A}\mathrm{e}(u_0):\mathrm{e}(v)\right)}-\mathrm{A}\mathrm{e}(u_0):\nabla{\left(\nabla{v}\theta\right)}=0,
$$ \textit{a.e.} on $\Omega_0$, one deduces using the divergence formula that
\begin{multline*}
\int_{\Omega_0}\left(\mathrm{div}\left(\mathrm{A}\mathrm{e}(u_0)\right)\theta^{\top}+\left(\mathrm{A}\mathrm{e}(u_0)\right)\nabla{\theta}^{\top}+\mathrm{A}\left(\nabla{u_0}\nabla{\theta}\right)-\mathrm{A}\mathrm{e}\left(\nabla{u_0}\theta\right)\right):\nabla{v}\\-\int_{\Omega_0}\mathrm{div}(\theta)\mathrm{A}\mathrm{e}(u_0):\mathrm{e}\left(v\right)-\int_{\Gamma_0}\nabla{\theta}^{\top}\left(\mathrm{A}\mathrm{e}(u_0)\nn\right)\cdot v-\int_{\Gamma_{0}}\left(\theta\cdot\nn\right) \mathrm{div}\left(\mathrm{A}\mathrm{e}(u_0)\right)\cdot v\\=\int_{\Gamma_0}\theta\cdot\nn\left(-\mathrm{A}\mathrm{e}(u_0):\mathrm{e}(v)-\mathrm{div}\left(\mathrm{A}\mathrm{e}(u_0)\right)\cdot v\right)+\nabla{v}^{\top}(\mathrm{A}\mathrm{e}(u_0)\nn)\cdot\theta-\nabla{\theta}^{\top}(\mathrm{A}\mathrm{e}(u_0)\nn)\cdot v,
\end{multline*}
for all $v\in\mathcal{C}^{\infty}(\overline{\Omega_{0}},\R^d)$. Furthermore, since $\Gamma_0$ is of class $\mathcal{C}^{3}$ and $u_0\in\HH^3(\Omega_0,\R^d)$, $\mathrm{A}\mathrm{e}(u_0)\nn$ can be extended into a function defined in $\Omega_{0}$ such that~$\mathrm{A}\mathrm{e}(u_0)\nn\in\HH^{2}(\Omega_{0},\R^d)$. Thus, it holds that~$\mathrm{A}\mathrm{e}(u_0)\nn\cdot v\in\mathrm{W}^{2,1}(\Omega_{0},\R^d)$, for all $v\in\mathcal{C}^{\infty}(\overline{\Omega_{0}},\R^d)$, and one can use Proposition~\ref{intbord} to get that
\begin{multline*}
\int_{\Gamma_0}\theta\cdot\nn\left(-\mathrm{A}\mathrm{e}(u_0):\mathrm{e}(v)-\mathrm{div}\left(\mathrm{A}\mathrm{e}(u_0)\right)\cdot v\right)+\nabla{v}^{\top}(\mathrm{A}\mathrm{e}(u_0)\nn)\cdot\theta-\nabla{\theta}^{\top}(\mathrm{A}\mathrm{e}(u_0)\nn)\cdot v\\=\int_{\Gamma_0}\theta\cdot\nn\left(-\mathrm{A}\mathrm{e}(u_0):\mathrm{e}(v)-\mathrm{div}\left(\mathrm{A}\mathrm{e}(u_0)\right)\cdot v+\partial_{\nn}\left(\mathrm{A}\mathrm{e}(u_0)\nn\cdot v\right)+H\mathrm{A}\mathrm{e}(u_0)\nn\cdot v\right)
\\-\int_{\Gamma_0}\left(\nabla{\left(\mathrm{A}\mathrm{e}(u_0)\nn\right)}\theta+\nabla{\theta}^{\top}(\mathrm{A}\mathrm{e}(u_0)\nn)+\mathrm{div}_{\tau}(\theta)\mathrm{A}\mathrm{e}(u_0)\nn\right)\cdot v,
\end{multline*}
for all~$v\in\mathcal{C}^{\infty}(\overline{\Omega_{0}},\R^d)$.
Moreover, from Proposition~\ref{beltrami} it follows that
\begin{equation*}
\int_{\Gamma_0}\theta\cdot\nn\left(-\mathrm{div}\left(\mathrm{A}\mathrm{e}(u_0)\right) +H\mathrm{A}\mathrm{e}(u_0)\nn\right)\cdot v = \int_{\Gamma_0}\mathrm{A}\mathrm{e}(u_0):\nabla_{\tau}\left(v\left(\theta\cdot\nn\right)\right)-\left(\theta\cdot\nn\right)\partial_{\nn}\left(\mathrm{A}\mathrm{e}(u_0)\right)\nn\cdot v,
\end{equation*}
for all~$v\in\mathcal{C}^{\infty}(\overline{\Omega_{0}},\R^d)$. Therefore, using the following equalities
$$\mathrm{A}\mathrm{e}(u_0):\nabla_{\tau}\left(v\left(\theta\cdot\nn\right)\right)=\theta\cdot\nn\left(\mathrm{A}\mathrm{e}(u_0):\nabla_{\tau}v\right)+\mathrm{A}\mathrm{e}(u_0)\nabla_{\tau}\left(\theta\cdot\nn\right)\cdot v,
$$
\textit{a.e.} on $\Gamma_0$, and
$$\mathrm{A}\mathrm{e}(u_0):\nabla_{\tau}v=\mathrm{A}\mathrm{e}(u_0):e(v)-\nabla{v}^{\top}(\mathrm{A}\mathrm{e}(u_0)\nn)\cdot\nn, 
$$ 
\textit{a.e.} on $\Gamma_0$,
it holds that
\begin{multline*}
    \int_{\Gamma_0}\theta\cdot\nn\left(-\mathrm{A}\mathrm{e}(u_0):\mathrm{e}(v)-\mathrm{div}\left(\mathrm{A}\mathrm{e}(u_0)\right)\cdot v+\partial_{\nn}\left(\mathrm{A}\mathrm{e}(u_0)\nn\cdot v\right)+H\mathrm{A}\mathrm{e}(u_0)\nn\cdot v\right)
\\-\int_{\Gamma_0}\left(\nabla{\left(\mathrm{A}\mathrm{e}(u_0)\nn\right)}\theta+\nabla{\theta}^{\top}(\mathrm{A}\mathrm{e}(u_0)\nn)+\mathrm{div}_{\tau}(\theta)\mathrm{A}\mathrm{e}(u_0)\nn\right)\cdot v \\ =\int_{\Gamma_0}\left(\theta \cdot \nn\left(\partial_{\nn}\left(\mathrm{A}\mathrm{e}(u_0)\nn\right)-\partial_{\nn}\left(\mathrm{A}\mathrm{e}(u_0)\right)\nn\right)+\mathrm{A}\mathrm{e}(u_0)\nabla_{\tau}\left(\theta\cdot\nn\right)\right)\cdot v\\
-\int_{\Gamma_0}\left(\nabla{\left(\mathrm{A}\mathrm{e}(u_0)\nn\right)}\theta+\left(\mathrm{div}_{\tau}(\theta)\mathrm{I}+\nabla{\theta}^{\top}\right)\mathrm{A}\mathrm{e}(u_0)\nn\right)\cdot v,
\end{multline*}
for all~$v\in\mathcal{C}^{\infty}(\overline{\Omega_{0}},\R^d)$. Thus,
\begin{multline*}
    \int_{\Omega_0}\left(\mathrm{div}\left(\mathrm{A}\mathrm{e}(u_0)\right)\theta^{\top}+\left(\mathrm{A}\mathrm{e}(u_0)\right)\nabla{\theta}^{\top}+\mathrm{A}\left(\nabla{u_0}\nabla{\theta}\right)-\mathrm{A}\mathrm{e}\left(\nabla{u_0}\theta\right)\right):\nabla{v}\\-\int_{\Omega_0}\mathrm{div}(\theta)\mathrm{A}\mathrm{e}(u_0):\mathrm{e}\left(v\right)-\int_{\Gamma_{0}}\nabla{\theta}^{\top}\left(\mathrm{A}\mathrm{e}(u_0)\nn\right)\cdot v-\int_{\Gamma_0}\left(\theta\cdot\nn\right) \mathrm{div}\left(\mathrm{A}\mathrm{e}(u_0)\right)\cdot v\\=\int_{\Gamma_0}\left(\theta \cdot \nn\left(\partial_{\nn}\left(\mathrm{A}\mathrm{e}(u_0)\nn\right)-\partial_{\nn}\left(\mathrm{A}\mathrm{e}(u_0)\right)\nn\right)+\mathrm{A}\mathrm{e}(u_0)\nabla_{\tau}\left(\theta\cdot\nn\right)\right)\cdot v\\-\int_{\Gamma_0}\left(\nabla{\left(\mathrm{A}\mathrm{e}(u_0)\nn\right)}\theta+\left(\mathrm{div}_{\tau}(\theta)\mathrm{I}+\nabla{\theta}^{\top}\right)\mathrm{A}\mathrm{e}(u_0)\nn\right)\cdot v,
\end{multline*}
for all~$v\in\mathcal{C}^{\infty}(\overline{\Omega_{0}},\R^d)$, and one deduces by density of $\mathcal{C}^{\infty}(\overline{\Omega_{0}},\R^d)$ in $\HH^1(\Omega_0,\R^{d})$ that 
\begin{multline*}
\dual{u'_{0}}{\varphi-u'_{0} }_{\HH^{1}_{\mathrm{D}}(\Omega_0,\R^d)}\geq \int_{\Gamma_{\mathrm{S}_0}}\left(\theta \cdot \nn\left(\partial_{\nn}\left(\mathrm{A}\mathrm{e}(u_0)\nn\right)-\partial_{\nn}\left(\mathrm{A}\mathrm{e}(u_0)\right)\nn\right)+\mathrm{A}\mathrm{e}(u_0)\nabla_{\tau}\left(\theta\cdot\nn\right)\right)\cdot\left(\varphi-u'_{0}\right)\\
-\int_{\Gamma_{\mathrm{S}_0}}\left(\nabla{\left(\mathrm{A}\mathrm{e}(u_0)\nn\right)}\theta+\sigma_{\nn}(u_0)\left(\mathrm{div}_{\tau}(\theta)\mathrm{I}+\nabla{\theta}^{\top}\right)\nn\right)\cdot \left(\varphi-u'_{0}\right),
\end{multline*}
for all $\varphi\in\mathrm{T}_{\mathcal{K}^{1}(\Omega_0)}(u_0)\cap\left(\R\left(E_0-u_0\right)\right)^{\perp}+\nabla{\theta}u_0-\nabla{u_0}\theta$, which concludes the proof from Subsection~\ref{SectionSignorinicasscalairesansu}. 
\end{proof}

\begin{myRem}\normalfont
    Note that $\overline{u}'_0$ and $u'_0$ are not linear with respect to the direction~$\theta$. This nonlinearity is standard in shape optimization for variational inequalities (see, e.g.,~\cite{ABCJ,HINTERMULLERLAURAIN} or~\cite[Section 4]{SOKOZOL}).
\end{myRem}
\subsection{Shape gradient of the Signorini energy functional}\label{energyfunctional}

Consider the framework of Subsection~\ref{setting}. Thanks to the characterization of the material and shape derivatives obtained in Subsection~\ref{sectionepidiff}, we are now in a position to prove the shape differentiability of the Signorini energy functional~\eqref{energysigno}. 

\begin{myTheorem}\label{shapederivofJsigno1}
 The Signorini energy functional~$\mathcal{J}$ admits a shape gradient at $\Omega_{0}$ in the direction~$\theta\in\mathcal{C}_{\mathrm{D}}^{2,\infty}(\R^{d},\R^{d})$ given by
\begin{multline}\label{firstshapegradJJ}
    \mathcal{J}'(\Omega_{0})(\theta)=\int_{\Omega_0}\mathrm{div}\left(\theta\right)\frac{\mathrm{A}\mathrm{e}\left(u_0\right):\mathrm{e}\left(u_0\right)}{2}-\int_{\Omega_0}\mathrm{div}\left(\mathrm{A}\mathrm{e}\left(u_0\right)\right)\cdot\nabla{u_0}\theta-\int_{\Omega_0}\mathrm{A}\mathrm{e}\left(u_0\right):\nabla{u_0}\nabla{\theta}\\-\int_{\Gamma_{\mathrm{S}_0}}\theta\cdot\nn \left(f\cdot u_0\right)+\dual{\mathrm{A}\mathrm{e}(u_0)\nn}{\nabla{\theta}u_0}_{\HH^{-1/2}(\Gamma_0,\R^d)\times\HH^{1/2}(\Gamma_0,\R^d)}.
\end{multline}
\end{myTheorem}
\begin{proof}
    Since, 
    \begin{multline*}
             \mathcal{J}(\Omega_t)=\mathcal{J}((\mathrm{id}+t\theta)(\Omega_{0}))=-\frac{1}{2}\int_{\Omega_t}\mathrm{A}\mathrm{e}\left(u_t\right):\mathrm{e}\left(u_t\right)\\=-\frac{1}{2}\int_{\Omega_0}\mathrm{J}_{t}\mathrm{A}\left[\nabla{\overline{u}_t}\left(\mathrm{I}+t\nabla{\theta}\right)^{-1}\right]:\nabla{\overline{u}_t}\left(\mathrm{I}+t\nabla{\theta}\right)^{-1},
    \end{multline*}
then using the differential results~\ref{diff11},~\ref{diff12} and Theorem~\ref{materialderiv1}, one gets that
$$
\mathcal{J}(\Omega_t)=\mathcal{J}(\Omega_0)+t\mathcal{J}'(\Omega_{0})(\theta)+o(t),
$$
where $\mathcal{J}'(\Omega_{0})(\theta)$ is the directional derivative of $\mathcal{J}$ at $\Omega_0$ in the direction $\theta$ given by
\begin{equation*}
\mathcal{J}'(\Omega_{0})(\theta)=-\frac{1}{2}\int_{\Omega_{0}}\mathrm{div}(\theta)\mathrm{A}\mathrm{e}(u_0):\mathrm{e}(u_0)+\int_{\Omega_{0}}\mathrm{A}\mathrm{e}(u_0):\nabla{u_{0}}\nabla{\theta}-\dual{\overline{u}'_0}{u_{0}}_{_{\HH^{1}_{\mathrm{D}}(\Omega_0,\R^d)}}.
\end{equation*}
Moreover, since $\overline{u}'_0\pm u_0\in\mathrm{T}_{\mathcal{K}^{1}(\Omega_0)}(u_0)\cap\left(\R\left(E_0-u_0\right)\right)^{\perp}+\nabla{\theta}u_0$, one deduces from the variational formulation of $\overline{u}'_0$ (see Inequality~\eqref{inequalityofmaterialderiv1}) and the divergence formula that
\begin{multline*}
    \dual{\overline{u}'_0}{u_0 }_{\HH^{1}_{\mathrm{D}}(\Omega_0,\R^d)}=\int_{\Omega_0}\left(\mathrm{div}\left(\mathrm{A}\mathrm{e}(u_0)\right)\theta^{\top}+\left(\mathrm{A}\mathrm{e}(u_0)\right)\nabla{\theta}^{\top}+\mathrm{A}\left(\nabla{u_0}\nabla{\theta}\right)-\mathrm{div}(\theta)\mathrm{A}\mathrm{e}(u_0)\right):\nabla{u_0}\\+\int_{\Gamma_{\mathrm{S}_0}}\theta\cdot\nn (f\cdot u_0)-\dual{\mathrm{A}\mathrm{e}(u_0)\nn}{\nabla{\theta}u_0}_{\HH^{-1/2}(\Gamma_0,\R^d)\times\HH^{1/2}(\Gamma_0,\R^d)}.
\end{multline*}
Using the equality $\mathrm{div}\left(\mathrm{A}\mathrm{e}(u_0)\right)\theta^{\top}:\nabla{u_0}=\mathrm{div}\left(\mathrm{A}\mathrm{e}(u_0)\right)\cdot\nabla{u_0}\theta$ \textit{a.e.} on $\Omega_0$, it follows that~$\mathcal{J}'(\Omega_{0})(\theta)$ is given by Equality~\eqref{firstshapegradJJ}. Finally, note that~$\mathcal{J}'(\Omega_{0})$ is linear and continuous on~$\mathcal{C}_{\mathrm{D}}^{2,\infty}(\R^{d},\R^{d})$, thus~$\mathcal{J}'(\Omega_{0})$ is the shape gradient of $\mathcal{J}$ at $\Omega_{0}$.
\end{proof}

As we did for the material derivative, the presentation of Theorem~\ref{shapederivofJsigno1} can be improved under additional assumption.

\begin{myCor}\label{shapederivofJ}
Assume that $u_0\in\HH^2(\Omega_0,\R^d)$. Then the Signorini energy functional $\mathcal{J}$ admits a shape gradient at $\Omega_{0}$ in the direction~$\theta\in\mathcal{C}_{\mathrm{D}}^{2,\infty}(\R^{d},\R^{d})$ given by
$$ \mathcal{J}'(\Omega_{0})(\theta)=\int_{\Gamma_{\mathrm{S}_{0}}} \left(\theta\cdot\nn\left(\frac{\mathrm{A}\mathrm{e}(u_0):\mathrm{e}(u_0)}{2}-f\cdot u_{0}\right)+\mathrm{A}\mathrm{e}(u_0)\nn\cdot\left(\nabla{\theta}u_0-\nabla{u_0}\theta\right)\right).
$$
\end{myCor}
\begin{proof}
Let $\theta\in\mathcal{C}_{\mathrm{D}}^{2,\infty}(\R^{d},\R^{d})$. Since $u_0\in\HH^2(\Omega_0,\R^d)$, it follows from Theorem~\ref{shapederivofJsigno1} that 
\begin{multline*}
    \mathcal{J}'(\Omega_{0})(\theta)=-\frac{1}{2}\int_{\Omega_{0}}\theta\cdot\nabla{{\left(\mathrm{A}\mathrm{e}(u_0):\mathrm{e}(u_0)\right)}}+\int_{\Gamma_{0}}\theta\cdot\nn\frac{\mathrm{A}\mathrm{e}(u_0):\mathrm{e}(u_0)}{2}+\int_{\Omega_0}\mathrm{A}\mathrm{e}\left(u_0\right):\mathrm{e}\left(\nabla{u_0}\theta\right)\\-\int_{\Gamma_0}\mathrm{A}\mathrm{e}\left(u_0\right)\nn\cdot\nabla{u_0}\theta-\int_{\Omega_0}\mathrm{A}\mathrm{e}\left(u_0\right):\nabla{u_0}\nabla{\theta}-\int_{\Gamma_{\mathrm{S}_0}}\theta\cdot\nn \left(f\cdot u_0\right)+\int_{\Gamma_{\mathrm{S}_0}}\mathrm{A}\mathrm{e}(u_0)\nn\cdot\nabla{\theta}u_0.
\end{multline*}
Moreover, since
$$
\mathrm{A}\mathrm{e}(u_0):\mathrm{e}\left(\nabla{u_0}\theta\right)=\mathrm{A}\mathrm{e}(u_0):\nabla{u_0}\nabla{\theta}+\frac{1}{2}\theta\cdot\nabla{{\left(\mathrm{A}\mathrm{e}(u_0):\mathrm{e}(u_0)\right)}} \text{ \textit{a.e.} on } \Omega_0,
$$
it holds that
\begin{multline*}
     \mathcal{J}'(\Omega_{0})(\theta)=\int_{\Gamma_{0}}\theta\cdot\nn\left(\frac{\mathrm{A}\mathrm{e}(u_0):\mathrm{e}(u_0)}{2}\right)-\int_{\Gamma_0}\mathrm{A}\mathrm{e}(u_0)\nn\cdot\nabla{u_0}\theta-\int_{\Gamma_{\mathrm{S}_0}}\theta\cdot\nn (f\cdot u_0)+\int_{\Gamma_0}\mathrm{A}\mathrm{e}(u_0)\nn\cdot\nabla{\theta}u_0,
\end{multline*}
which completes the proof since $\theta=0$ on $\Gamma_{\mathrm{D}}$.
\end{proof}

\begin{myRem}\label{remarkadjoint}\normalfont
Consider the framework of Theorem~\ref{shapederivofJsigno1}. It is interesting to note that the scalar product~$\dual{\overline{u}'_0}{u_0 }_{\HH^{1}_{\mathrm{D}}(\Omega_0,\R^{d})}$ is linear with respect to the direction $\theta$, while $\overline{u}'_0$ is not. This leads to an expression of the shape gradient~$\mathcal{J}'(\Omega_{0})(\theta)$ that is linear with respect to the direction~$\theta$, thus to the shape differentiability of the Signorini energy functional $\mathcal{J}$ at $\Omega_{0}$. Note that, in the context of cracks and variational
inequalities involving unilateral conditions, it can already be observed that the shape gradient of the energy functional is linear with respect to $\theta$ (see, e.g.,~\cite[Theorem 2.22 or Theorem 4.20]{Gilles}). Furthermore note that the shape gradient~$\mathcal{J}'(\Omega_{0})(\theta)$ depends only on~$u_0$ (and not on $u’_0$), therefore its expression is explicit with respect to the direction $\theta$, and there is no need to introduce any adjoint problem to perform numerical simulations. Nevertheless, for other cost functionals, some difficulties can appear to correctly define an adjoint problem due to nonlinearities in shape gradients, and may constitute an interesting area for future researches.
\end{myRem}

\begin{myRem}
     Note that an expression of the shape gradient of the Signorini energy functional has been obtained in~\cite[Section 5.5]{Gilles} in the particular case where $\Gamma_{\mathrm{S_0}}$ is a rectilinear boundary part. Therefore, Theorem~\ref{shapederivofJsigno1} can be seen as a generalisation of this result.
\end{myRem}

\begin{comment}
    \textcolor{blue}{\begin{myRem}
    Instead of the Signorini energy functional, it would be possible to consider the compliance defined by
     $$
     \mathcal{C}(\Omega) := \int_{\Omega}f\cdot u_{\Omega},
     $$
     for all $\Omega\in\mathcal{U}_{\mathrm{ref}}$,
     where $u_\Omega \in\HH^{1}_{\mathrm{D}}(\Omega,\R^{d})$ is the unique solution to the Signorini problem~\eqref{Signoriniproblem2221}. Indeed, from the weak formulation satisfies by $u_\Omega$, it follows that $\int_{\Omega}f\cdot u_{\Omega}=\int_{\Omega}\mathrm{A}\mathrm{e}\left(u_\Omega\right):\mathrm{e}\left(u_\Omega\right)$, hence $\mathcal{C}(\Omega)=-2\mathcal{J}(\Omega)$ and the shape gradient of $\mathcal{C}$ follows from the one of $\mathcal{J}$.
\end{myRem}}
\end{comment}

\section{Numerical simulations}\label{numericalsim}
In this section we numerically solve an example of the shape optimization problem~\eqref{shapeOptim} in the two-dimensional case $d=2$, by making use of our theoretical results obtained in Section~\ref{mainresultoff}. The numerical simulations have been performed using Freefem++ software~\cite{HECHT} with P1-finite elements and standard affine mesh. We could use the expression of the shape gradient of~$\mathcal{J}$ obtained in Theorem~\ref{shapederivofJsigno1} but, in order to simplify the computations, we chose to use the expression provided in Corollary~\ref{shapederivofJ} under the additional assumption~$u_0 \in \mathrm{H}^2(\Omega_0,\R^d)$ that we assumed to be true at each iteration.

\subsection{Numerical methodology}\label{methodnum}
Consider an initial shape~$\Omega_0 \in \mathcal{U}_{\mathrm{ref}}$. Note that Corollary~\ref{shapederivofJ} allows to exhibit a descent direction~$\theta_0$ of the Signorini energy functional~$\mathcal{J}$ at~$\Omega_0$ by finding the solution $\theta_0$ to the variational equality
\begin{equation*}
\dual{\theta_0}{\theta}_{\HH^{1}_{\mathrm{D}}(\Omega_0,\R^d)}=-\mathcal{J}'(\Omega_0)(\theta), \qquad \forall\theta\in\HH^{1}_{\mathrm{D}}(\Omega_0,\R^d),
\end{equation*}
since it satisfies~$\mathcal{J}'(\Omega_{0})(\theta_0)=-\left\|\theta_0\right\|^{2}_{\HH^{1}_{\mathrm{D}}(\Omega_0,\R^d)}\leq0
$.

In order to numerically solve the shape optimization problem~\eqref{shapeOptim} on a given example, we have to deal with the volume constraint~$\vert \Omega \vert = \vert \Omega_{\mathrm{ref}} \vert>0$. For this purpose, the Uzawa algorithm (see, e.g.,~\cite[Chapter 3 p.64]{ALL}) is used, and one refers to~\cite[Section 4]{ABCJ} for methodological details. 

Let us mention that the Signorini problem is numerically solved using the Nitsche's method (see, e.g.,~\cite{BRETIN,CHOULY,NITSCHE}). In a nutshell, the solution $u_0\in\HH^{1}_{\mathrm{D}}(\Omega_0,\R^d)$ is approximated by $u_{0}^h\in\mathbb{V}^h$ which is the solution to the Nitsche's formulation
\begin{multline*}
        \int_{\Omega_0}\mathrm{A}\mathrm{e}(u_{0}^h):\mathrm{e}(v^h)-\gamma\int_{\Gamma_{\mathrm{S}_0}}\sigma_{\nn}(u_{0}^h)\sigma_{\nn}(v^h)+\frac{1}{\gamma}\int_{\Gamma_{\mathrm{S}_0}}\left[{u_0}_{\nn}^{h}-\gamma\sigma_{\nn}(u_{0}^h)\right]_{+}\left[v_{\nn}^h-\gamma\sigma_{\nn}(v^h)\right]\\=\int_{\Omega_0}f\cdot v^h, \qquad\forall v^h\in\mathbb{V}^h,
\end{multline*}
where $\mathbb{V}^h$ is the classical P1-finite elements space whose elements are null on $\Gamma_{\mathrm{D}}$ (see~\cite{CHOULY} for numerical analysis details). We also precise that, for all~$i\in\mathbb{N}^{*}$, the difference between the Signorini energy functional $\mathcal{J}$ at the iteration $20\times i$ and at the iteration~$20\times (i-1)$ is computed. The smallness of this difference is used as a stopping criterion for the algorithm.

\subsection{Two-dimensional example and numerical results}
In this subsection, let~$d=2$ and $f\in\HH^{1}(\R^{2},\R^2)$ given by
$$
\fonction{f}{\R^{2}}{\R^2}{(x,y)}{\displaystyle f(x,y):= \begin{pmatrix}
\frac{1}{2}\exp(x^2)\eta(x,y) & 0
\end{pmatrix},}
$$
where~$\eta\in\mathcal{C}^{\infty}_0(\R^2,\R)$ is a cut-off function chosen such that $\eta=1$ everywhere inside a ball of extremely large radius, and $0$ outside a slightly larger ball, so that $f$ belongs to $\HH^1(\R^2,\R^2)$. The reference shape $\Omega_{\mathrm{ref}}$ is the unit disk of $\R^{2}$, and the fixed part $\Gamma_{\mathrm{D}}$ is given by
$$
\Gamma_{\mathrm{D}}=\left\{\left(\cos{\alpha},\sin{\alpha}\right)\in\Gamma_{\mathrm{ref}} \mid \alpha\in\left[\frac{\pi}{6},\frac{5\pi}{6}\right]\cup\left[\frac{7\pi}{6},\frac{11\pi}{6}\right] \right \},
$$
(see Figure~\ref{figurrre}). The volume constraint is $\vert \Omega_{\mathrm{ref}} \vert=\pi$
 and the initial shape is $\Omega_{0}:=\Omega_{\mathrm{ref}}$. 

 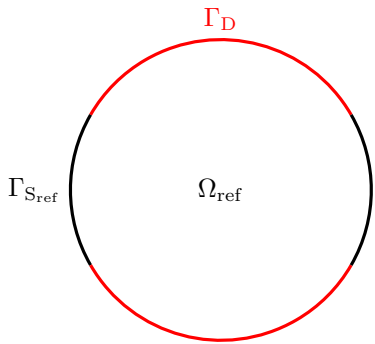
\begin{figure}[ht]
    \centering
\begin{tikzpicture}\label{figure1}
\draw (0,0) node{$\Omega_{\mathrm{ref}}$};
\draw [color=red, very thick] (1.732,1) arc (30:150:2);
\draw [color=black, very thick](-1.732,1) arc(150:210:2);
\draw [color=red, very thick](-1.732,-1) arc(210:330:2);
\draw [color=black, very thick](1.732,-1) arc(330:390:2);
\draw (0,2) [color=red] node[above]{$\Gamma_{\mathrm{D}}$};
\draw (-2,0) [color=black] node[left]{$\Gamma_{\mathrm{S}_{\mathrm{ref}}}$};
\end{tikzpicture}
    \caption{Unit disk $\Omega_{\mathrm{ref}}$ and its boundary $\Gamma_{\mathrm{ref}}=\Gamma_{\mathrm{D}}\cup\Gamma_{\mathrm{S}_{\mathrm{ref}}}$.}\label{figurrre}
\end{figure}

 We assume that, for all $\Omega\in\mathcal{U}_{\mathrm{ref}}$, the Cauchy stress tensor~$\sigma$, defined by~$\sigma(v):=\mathrm{A}\mathrm{e}(v)$ for all~$v\in\HH^{1}_{\mathrm{D}}(\Omega,\R^2)$, satisfies
 $$
 \sigma(v)=2\mu\mathrm{e}(v)+\lambda \mathrm{tr}\left(\mathrm{e}(v)\right)\mathrm{I},
 $$
 where $\mathrm{tr}\left(\mathrm{e}(v)\right)$ is the trace of the matrix $\mathrm{e}(v)$, and $\mu\geq0,\lambda\geq0$ are Lamé parameters (see, e.g.,~\cite{SALEN}). From a physical point of view, this assumption corresponds to \textit{isotropic} elastic solids. In the sequel, we consider the data $\mu=0.3846$ and~$\lambda=0.5769$, corresponding to a Young's modulus equal to~$1$ and to a Poisson's ratio equal to $0.3$, which is a typical value for a large variety of materials. One presents the numerical results obtained for this two-dimensional example using the methodology described in Subsection~\ref{methodnum}.
 
 In Figure~\ref{fig1} is represented the initial shape (left) and the shape which solves Problem~\eqref{shapeOptim} (right). On top are the vector values of the solution $u$ to the Signorini problem~\eqref{Signoriniproblem2221}. Note that the black boundary shows where $\sigma_{\nn}(u)=0$, while the yellow boundary shows where $u_\nn=0$. At the bottom is shown the values of the integrand of $\mathcal{J}$, i.e., the map $x\in\Omega\to \frac{1}{2}\mathrm{A}\mathrm{e}\left(u(x)\right):\mathrm{e}\left(u(x)\right)-f(x)\cdot u(x)$. It seems that the area where the integrand of $\mathcal{J}$ is the lowest (in orange) has been shifted to the left by "pushing" the left boundary (which corresponds to the part where there is no compressive stress), while in return, the right boundary (which corresponds to the contact part) has been pulled.
 
\begin{figure}[h!]
    \centering
    \includegraphics[scale=0.44, trim = 3.1cm 0cm 3.1cm 0.6cm, clip]{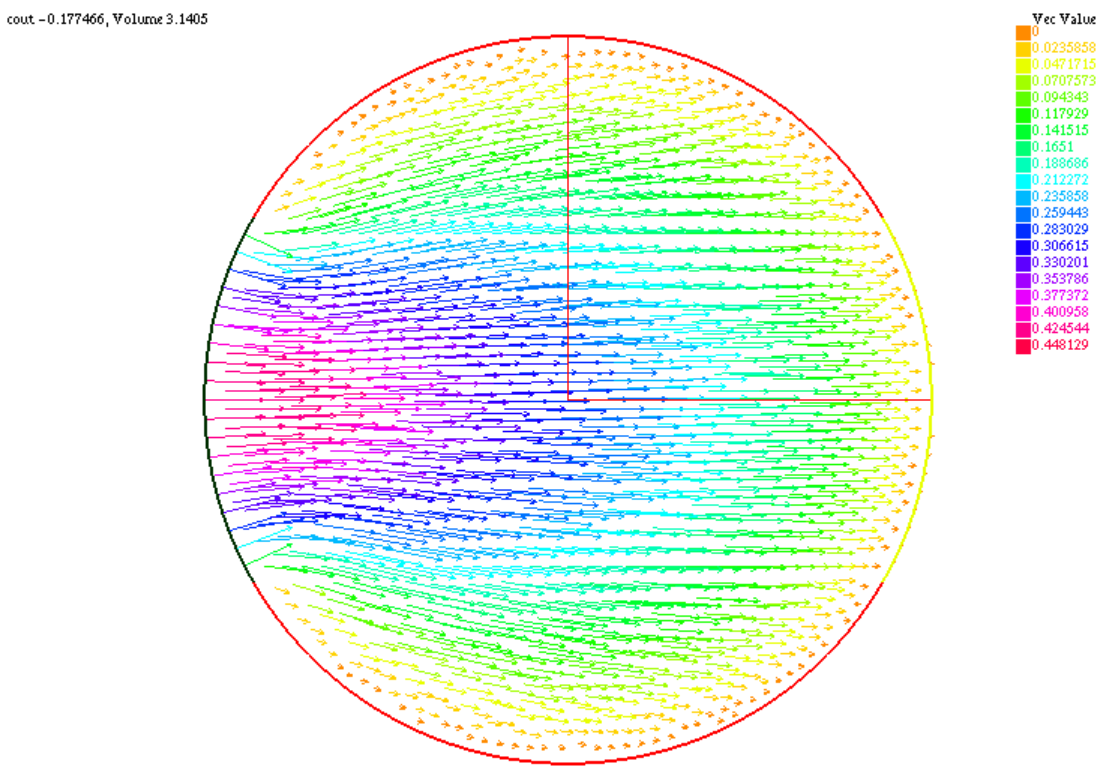}
    \includegraphics[scale=0.44,trim = 2.1cm 0cm 3.1cm 0.6cm, clip]{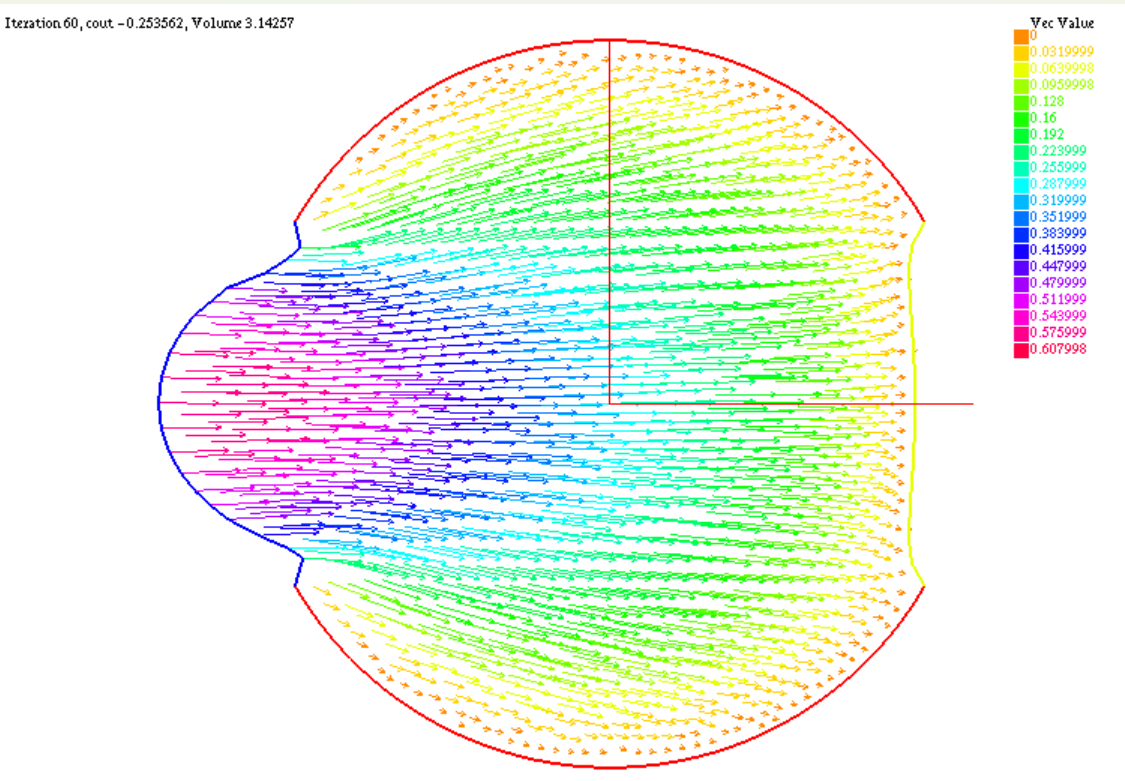}
    \includegraphics[scale=0.44,trim = 2.1cm 0cm 3.1cm 0.6cm, clip]{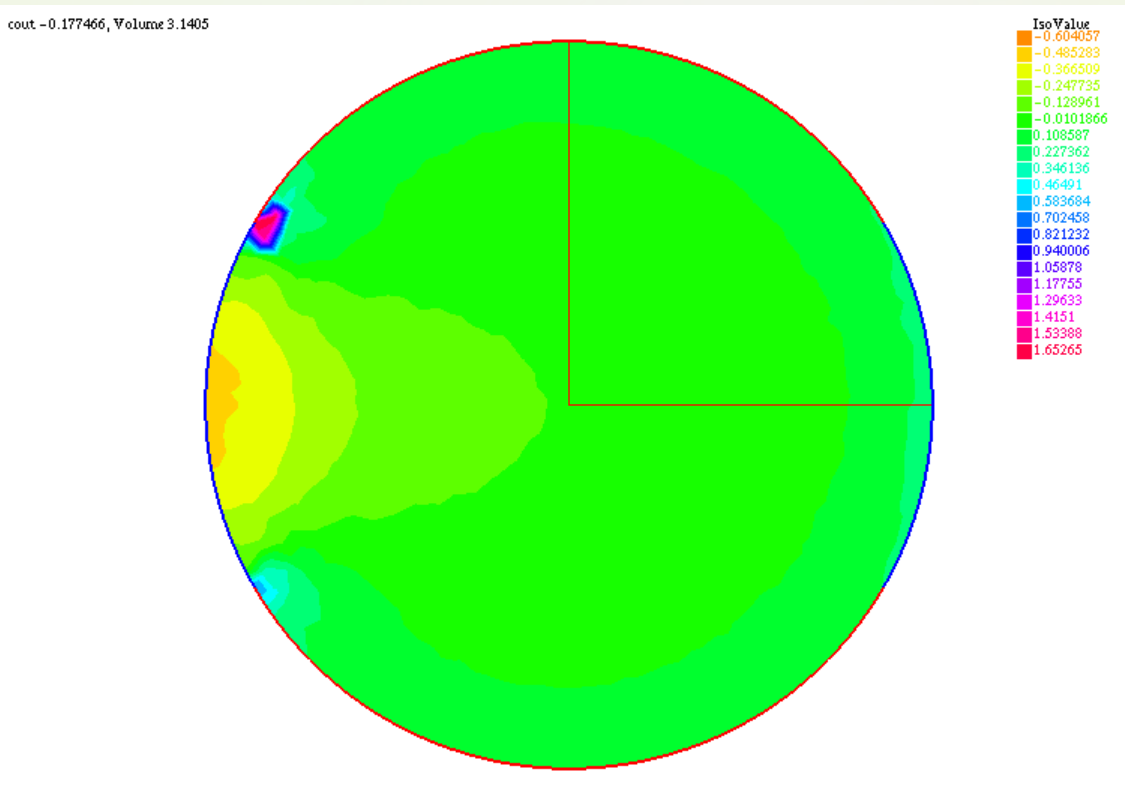}  
     \includegraphics[scale=0.44,trim = 2.1cm 0cm 3.1cm 0.6cm, clip]{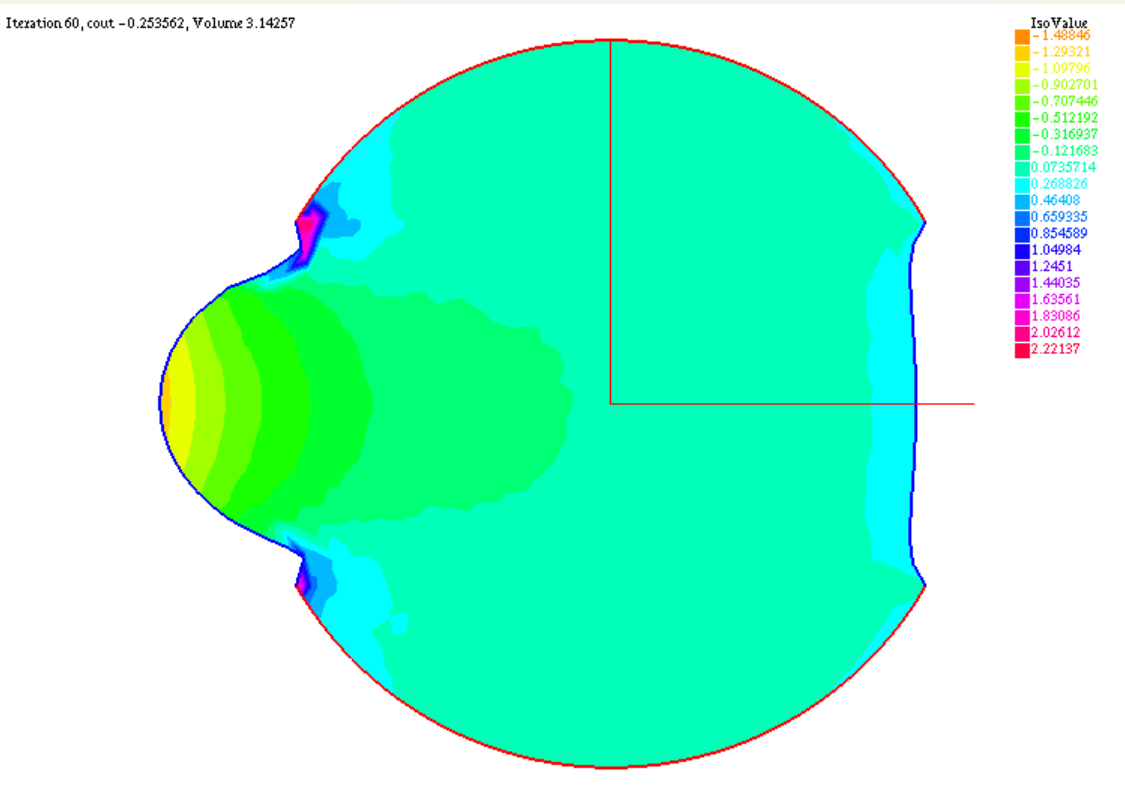} 
    \caption{Initial shape (left) and the shape minimizing $\mathcal{J}$ (right), under the volume constraint~$\vert \Omega_{\mathrm{ref}} \vert=\pi$. On top is shown the vector values of the Signorini solution, while at bottom is shown the values of the integrand of $\mathcal{J}$, i.e., the map $x\in\Omega\to \frac{1}{2}\mathrm{A}\mathrm{e}\left(u(x)\right):\mathrm{e}\left(u(x)\right)-f(x)\cdot u(x)$. }
    \label{fig1}
\end{figure}

 Figure~\ref{figure3} shows the values of $\mathcal{J}$ (left) and the volume of the shape (right) with respect to the iteration. We observe that $\mathcal{J}$ is lower at the final shape, than at the initial shape, with some oscillations due to the Lagrange multiplier in order to satisfy the volume constraint. 

\begin{figure}[h!]
    \centering
\includegraphics[scale=0.58]{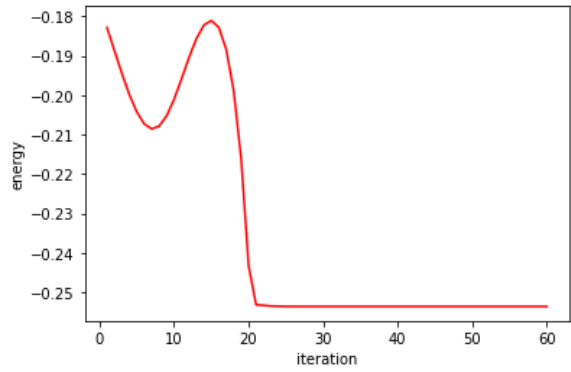}
    \hspace{2cm}
    \includegraphics[scale=0.58]{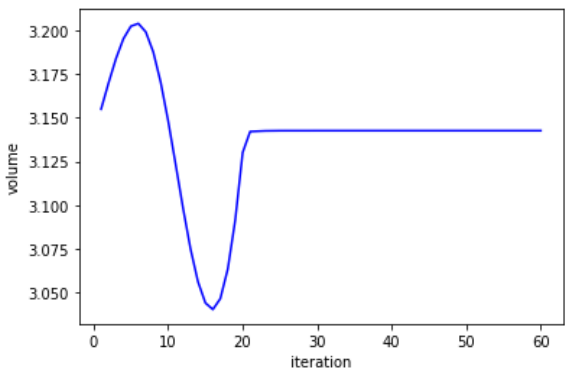}  
    \caption{Values of the energy functional (left) and the volume (right) with respect to the iterations.}
    \label{figure3}
\end{figure}

\appendix

\section{Reminders}\label{rappelgeneral}

In this appendix we recall in Subsection~\ref{capacity} some concepts from capacity theory and, in Subsection~\ref{rappelfunct}, some results on differential geometry.

\subsection{Notions from capacity theory}\label{capacity}
Let us recall some notions from capacity theory (we refer to standard references such as~\cite{DENY,HARAUX,HENROT,MIGNOT}). Let $d\in\mathbb{N}^{*}$ be a positive integer, $(\mathrm{X}, \mathcal{B}(\mathrm{X}), \xi )$ be a positively measured topological space with its borelian~$\sigma$-algebra,~$\xi$ a Radon measure, and~$\mathrm{X}\subset\R^d$ a locally compact set, admitting a countable compact covering. Let~$\mathcal{H}\subset\LL^2(\mathrm{X},\xi)$ be a vector space endowed with a scalar product~$\dual{\cdot}{\cdot}_{\mathcal{H}}$ and~$\left\|\cdot\right\|_{\mathcal{H}}$ the corresponding norm.
\begin{myDefn}
    Consider $\mathrm{B}\in\mathcal{B}(\mathrm{X})$ and let us introduce the closed convex subset
    $$
    \mathrm{C_{\mathrm{B}}}:=\left\{v\in\mathcal{H}\mid v\geq1 \text{ $\xi$-\textit{a.e.} on a neighborhood of }\mathrm{B} \right\}.
    $$
    The capacity of $\mathrm{B}$ is defined by
    $$
    \mathrm{cap}(\mathrm{B}):=   \left\|\mathrm{proj}_{\mathrm{C_{\mathrm{B}}}}(0)\right\|^{2}_{\mathcal{H}},
    $$
where $\mathrm{proj}_{\mathrm{C_{\mathrm{B}}}}$ is the projection operator onto the nonempty closed convex set $\mathrm{C_{\mathrm{B}}}$.
\end{myDefn}
\begin{myDefn}\label{quasieve}
  A property holds quasi everywhere (denoted \textit{\textit{q.e.}}) if it holds for all elements in a set except a subset of null capacity.
\end{myDefn}

\begin{myDefn}
    A function $v :\mathrm{X} \rightarrow \R$ is said to be quasi-continuous if there exists a decreasing sequence of open sets $(w_{n})_{n\in\mathbb{N}}$ such that $\mathrm{cap}(w_n)\rightarrow 0$ when~$n\rightarrow +\infty$ and $v_{\mid\mathrm{X} \backslash w_n}$ is continuous for all~$n\in\mathbb{N}$.
\end{myDefn}

Now, let us assume that~$(\mathcal{H},\dual{\cdot}{\cdot}_{\mathcal{H}})$ is a Dirichlet space (see, e.g.,~\cite{MIGNOT}). Then, one can prove the following proposition (see, e.g.,~\cite{HARAUX,HENROT,MIGNOT}).

\begin{myProp}\label{uniquerepres}
    For all $v\in\mathcal{H}$, there exists a unique quasi-continuous representative in the class of $v$ (for the $\textit{q.e.}$ equivalence relation).
\end{myProp}

To conclude, let us give two examples of Dirichlet space (see~\cite{MIGNOT} for the first example and~\cite[Chapter 4]{SOKOZOL} for the second one).

\begin{myExa}
Let $\Omega$ is a nonempty bounded connected open subset of $\R^{d}$ with a Lipschitz continuous boundary. Then $\mathcal{H}:=\HH^1(\Omega,\R)$ endowed with its standard scalar product $\dual{\cdot}{\cdot}_{\HH^1(\Omega,\R)}$ is a Dirichlet space.
\end{myExa}

\begin{myExa}\label{DirichletSpace}
Let $\Omega$ be a nonempty bounded connected open subset of $\R^{d}$ with a Lipschitz continuous boundary $\Gamma:=\partial{\Omega}$. Assume that $\Gamma$ is given by the decomposition~$\Gamma=\Gamma_{1}\cup\Gamma_{2}$, where~$\Gamma_{1}$ and $\Gamma_{2}$ are two measurable disjoint subsets of $\Gamma$. Then,
$$
\mathcal{H}:=\left\{ v\cdot\nn\in\HH^{1/2}(\Gamma_2,\R) \mid v\in\HH^1(\Omega,\R^d) \text{ and } v=0 \text{ \textit{a.e.} on } \Gamma_1 \right\},
$$
is a Dirichlet space endowed with the scalar product defined in~\cite[Chapter 4, Eq. (4.192) p.208]{SOKOZOL}, where $\nn$ is the outward-pointing unit normal vector to $\Gamma$.
\end{myExa}

\subsection{Reminders on differential geometry}\label{rappelfunct}
Let $d\in\mathbb{N}^{*}$ be a positive integer, $\Omega$ be a nonempty bounded connected open subset of~$\R^{d}$ with a Lipschitz-boundary $\Gamma :=\partial{\Omega}$ and $\nn$ the outward-pointing unit normal vector to $\Gamma$.
\medskip

The next proposition, known as divergence formula, can be found in~\cite[Theorem 4.4.7 p.104]{ALLNUM}.

\begin{myProp}[Divergence formula]\label{div}
If $w\in \HH_{\mathrm{div}}(\Omega, \R^{d\times d})$,
then $w$ admits a normal trace, denoted by~$w\nn \in \HH^{-1/2}(\Gamma,\R^d)$, satisfying
$$
\displaystyle\int_{\Omega}\mathrm{div}(w)\cdot v+\int_{\Omega}w:\nabla v=\dual{w\nn}{v}_{\HH^{-1/2}(\Gamma,\R^d)\times \HH^{1/2}(\Gamma,\R^d)}, \qquad\forall v \in \HH^1(\Omega,\R^{d}).
$$
\end{myProp}

The following propositions will be useful and their proofs can be found in~\cite{HENROT}.

\begin{myProp}\label{intbord}
Assume that~$\Gamma$ is of class $\mathcal{C}^2$ and let $\theta \in \mathcal{C}^{1}(\R^{d},\R^{d})$. It holds that
\begin{equation*}
    \int_{\Gamma}(\theta\cdot\nabla{v}+v\mathrm{div}_{\tau}(\theta))=\int_{\Gamma}\theta\cdot\nn(\partial_{\nn}v+Hv), \qquad\forall v \in \mathrm{W}^{2,1}(\Omega,\R),
\end{equation*}
where $\mathrm{div}_{\tau}(\theta):=\mathrm{div}(\theta)-(\nabla{ \theta}\nn \cdot \nn) \in \LL^\infty(\Gamma)$ is the {tangential divergence} of $\theta$,~$\partial_{\nn} v := \nabla v \cdot \nn \in \LL^1(\Gamma,\R)$ stands for the normal derivative of~$v$, and $H$ stands for the \textit{mean curvature} of $\Gamma$.
\end{myProp}

\begin{myProp}\label{beltrami}
Assume that $\Gamma$ is of class $\mathcal{C}^2$ and let $w\in\mathrm{H}^{2}(\Omega,\R^{d\times d})$. It holds that 
\begin{equation*}
    \mathrm{div}(w)=\mathrm{div}_{\tau}\left(w_{\tau}\right)+H w\nn+\left(\partial_{\nn}w\right)\nn \qquad \text{\textit{a.e.}\ on } \Gamma,
\end{equation*}
where $\mathrm{div}_{\tau}\left(w_{\tau}\right)\in\LL^2(\Gamma,\R^d)$ is the vector whose the $i$-th component is defined by $\left(\mathrm{div}_{\tau}\left(w_{\tau}\right)\right)_i:=\mathrm{div}_{\tau}((w_{i})_{\tau})\in\LL^2(\Gamma,\R)$, where $(w_{i})_{\tau}:=w_i-(w_i\cdot\nn) \nn\in\LL^2(\Gamma,\R^d)$ and~$w_i\in\HH^1(\Omega,\R^{d})$ is the transpose of the~$i$-th line of~$w$, and where~$\partial_{\nn}w\in\LL^2(\Gamma,\R^{d\times d})$ is the matrix whose the~$i$-th line is the transpose of the vector $\partial_{\nn}w_{i}:=(\nabla{w_{i}})\nn\in\LL^2(\Gamma,\R^d)$, for all~$i\in[[1,d]]$. Moreover it holds that
\begin{equation*}
\int_{\Gamma}v\cdot\mathrm{div}_{\tau}\left(w_{\tau}\right)=-\int_{\Gamma}w:\nabla_{\tau}v, \qquad \forall v \in \HH^{2}(\Omega,\R^d),
\end{equation*}
where $\nabla_{\tau}v$ is the matrix whose the $i$-th line is the transpose of the \textit{tangential gradient} $\nabla_{\tau}v_i:=\nabla{v_i}-(\partial_{\nn}v_i)\nn \in \HH^{1/2}(\Gamma,\R^{d})$, for all $i\in[[1,d]]$.
\end{myProp}

\bibliographystyle{abbrv}
\bibliography{biblio}

\end{document}